\documentclass{amsart}
\usepackage{graphicx}

\usepackage{amssymb}
\usepackage{mathtools}
\usepackage{kpfonts}
\usepackage[T1]{fontenc}
\usepackage{multicol}
\usepackage{mathrsfs}
\usepackage[dvipsnames]{xcolor}
\usepackage{multicol}
\usepackage{amsthm}
\usepackage{float}

\usepackage{bbding} 
\usepackage{tikz}
\usepackage{tikz-cd}
\usetikzlibrary{arrows,shapes,patterns,decorations.markings}
\usetikzlibrary{shapes.arrows}   
\usetikzlibrary{positioning}
\usepackage{mathrsfs}
\usepackage{amssymb}
\usepackage{mathtools}
\usepackage{chessfss}
\usepackage{xfrac}
\usepackage{setspace}

\usepackage{geometry}
\PassOptionsToPackage{obeyspaces}{url}
\usepackage[hyphens]{xurl}
\usepackage{algorithm}
\usepackage{algpseudocode}

\usepackage{subcaption}
\usepackage{float}

\definecolor{1}{rgb}{1,0.2,0.3}
\definecolor{2}{rgb}{0.1,0.3,0.5}
\definecolor{3}{rgb}{1,1,0}
\definecolor{4}{rgb}{255,255,255}

\usepackage[hyperref,
giveninits,
style=alphabetic,
backend=biber
]{biblatex}
\renewbibmacro{in:}{} 

\newtheorem{theorem}{Theorem}
\newtheorem{corollary}[theorem]{Corollary}
\newtheorem{lemma}[theorem]{Lemma}

\theoremstyle{definition}

\newtheorem{hint}{Hint}
\newtheorem{challenge}{Challenge}
\theoremstyle{remark}
\newtheorem*{remark}{Remark}

\definecolor{tr-pink}{RGB}{245,169,187}
\definecolor{tr-blue}{RGB}{91, 206, 250}

\newcommand{\polfont}[1]{\texttt{#1}}

\newcommand{\scfa}{.58} 
\newcommand{\triscfa}{.65} 
\newcommand{\hexscfa}{1}
\newcommand{\opacite}{.9}
\newcommand{\radis}{0pt}

\usepackage{LRMR_drawing_macros}

\usepackage[pdfpagemode=UseNone, pdfstartview={XYZ null null null}]{hyperref}
\hypersetup{
	colorlinks=true,
	citecolor=tr-pink,
	linkcolor=tr-blue,
	urlcolor=tr-pink
}

\makeatletter
\newcommand{\addresseshere}{%
	\enddoc@text\let\enddoc@text\relax
}
\makeatother
\usepackage{tcolorbox}
\begin{document}
	
	\title{Extremal fences with polyforms}

	\author[A. Langlois-Rémillard]{Alexis Langlois-R\'emillard}
	\address[A. Langlois-Rémillard]{Hausdorff Center for Mathematics, Endenicher Allee 62, 53115 Bonn, Germany}
	\email{alexis.langlois-remillard@tutanota.com}
	\author[M. N. Müßig]{Mia N. M\"u\ss{}ig}
	\address[M. Müßig ]{Ludwig-Maximilians-Universit\"at M\"unchen, Munich, Germany}
	\email{nienna@miamuessig.de}
	\author[\'E. Rold\'an]{\'Erika Rold\'an}
	\address[\'E. Rold\'an]{ Max Planck Institute for Mathematics in the Sciences, Inselstraße 22, Leipzig, Germany and ScaDS.AI, Leipzig, Humboldtstra\ss{}e 25, Leipzig, Germany}
	\email{roldan@mis.mpg.de}
	
	\begingroup 
	\let\MakeUppercase \relax 

	\begin{abstract}
	 We present results around an isoperimetric problem built on polyforms: What is the biggest enclosed area one can build using  polyforms in each of the three plane tessellations? We give challenges to the readers and present Shimauchi's proof of the biggest area a fence made of pentominoes can enclose. A translation of the instance using integer linear programming is also given. 
	
	A companion web app is available to test some of the challenges we propose and for activities, and we included extra pages with a cutout handout of the board and pieces of the puzzles, so that you can print and cut them to read this paper hands on.
	\end{abstract}
\maketitle
\endgroup
	
	\begin{center}
		\textit{To the memory of Sayan Mukherjee,}\\
		\textit{whose enthusiasm for the fence challenge and unwavering support made this research possible.}
	\end{center}\medskip

\begingroup
\centering
\begin{tcolorbox}[width=.87\textwidth]
		\noindent The following is an English version of the paper
		\begin{description}
			\item[\cite{LRMR25f}] \fullcite{LRMR25f}. 
		\end{description}	
		\noindent Since its publication, groups of students, researchers and enthusiasts have solved Challenges~\ref{challenge:hexomino_fence},~\ref{challenge:minmax}, and~\ref{challenge:pentocube}. We will update this arXiv version at the end of 2026 when a follow-up to this paper containing the solutions is published. In the meantime, feel free to try them! \medskip
\end{tcolorbox}
\endgroup
	
	\noindent \textbf{Hands-on first!} To enjoy the full potential of this paper, we invite you, the reader, to begin by printing and \hspace{0.2em}\ScissorRight\hspace{0.2em} the tetrominoes in Challenge~\ref{challenge:tetrominoes}, all the polyforms depicted in Figures~\ref{pentominoes}, \ref{fig:tetrahexes}, and~\ref{fig:hexiamonds}, along with the tessellation boards provided in Appendix~\ref{app:fences}.
	With these materials, you will be able to try the first round of challenges, before diving into the rest of the paper. Some of the answers are presented in the sections that follow;tetromino others are available upon request\footnote{Write to any of the authors.}.\medskip
	
	Prefer pixels to paper? Every challenge in this paper can also be
	played in your browser: interactive fence labs, one per tessellation,
	let you drag, rotate, flip and snap the pieces and watch the enclosed
	area grow. Build your fences at \url{https://erikaroldanroa.github.io/fence-challenge/}~\cite{FenceHub}.\medskip

	Although this paper presents new results in extremal discrete computational geometry and topology, along with several open problems, we have chosen to structure it in a way that invites participation from a broad audience, both inside and outside academia. The activities can be easily adapted for a variety of formats: from a five-minute stop at a science fair tent, to week-long workshops, or even yearlong projects at the high school or undergraduate level~\cite{LRRR24}. In later sections, we also include challenges that, while accessible to a wide audience, pose open problems suitable for graduate students and researchers, offering material for deeper study and future exploration.
				
	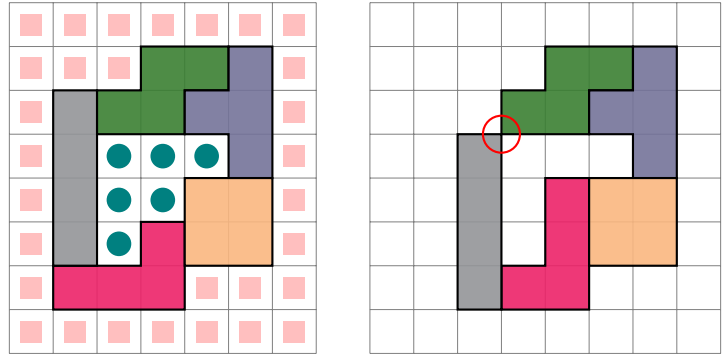
\begin{figure}[h]
		\centering
		\begin{tikzpicture}[scale=\scfa]
		\draw[help lines] (0,-1) grid (7,7);
		\pic[scale = \scfa,rotate=90] at (4,0){tetrominoL={1}};
		\pic[scale=\scfa,rotate=0] at (1,1) {tetrominoI};
		\pic[scale = \scfa,rotate =180] at (5,6) {tetrominoN={1}};
		\pic[scale = \scfa,rotate=270] at (4,6){tetrominoT};
		\pic[scale=\scfa,rotate=0] at (4,1) {tetrominoO};
		\foreach \x/\y in {2/1,2/2,2/3,3/2,3/3,4/3}
		{
			\fill[teal] (\x+.5,\y+.5) circle(8pt);
		}
		\foreach \x/\y in {0/-1,0/0,0/1, 0/2,0/3,0/4,0/5,0/6,
			1/-1,1/5,1/6,
			2/-1, 2/5,2/6,
			3/-1,3/6,
			4/-1,4/0,4/6, 
			5/-1,5/0,5/6,
			6/-1,6/0,6/6, 
			6/-1,6/0,6/1, 6/2,6/3,6/4,6/5,6/6}
		{
			\fill[pink] (\x+.5-.25,\y+.5-.25) rectangle (\x+.5+.25,\y+.5+.25);
		}
		\end{tikzpicture}\qquad 
		\begin{tikzpicture}[scale=\scfa]
		\draw[help lines] (-1,-1) grid (7,7);
		\pic[scale = \scfa,rotate=0] at (4,0){tetrominoL={-1}};
		\pic[scale=\scfa,rotate=0] at (1,0) {tetrominoI};
		\pic[scale = \scfa,rotate =180] at (5,6) {tetrominoN={1}};
		\pic[scale = \scfa,rotate=270] at (4,6){tetrominoT};
		\pic[scale=\scfa,rotate=0] at (4,1) {tetrominoO};
		\draw[red,thick] (2,4) circle(12pt) ;
		\end{tikzpicture}
		\caption{Left: A tetromino fence whose complement is clearly split into two polyominoes: the inside polyomino (marked by teal circles) and the outside polyomino (marked by pink squares). Right: In this case, the inside and outside polyominoes determined by the complement of the tetrominoes share a vertex, so this placement does not qualify as a fence under our definition.}
		\label{fig:def_fences}
	\end{figure}
	\section*{Warming-up challenges}

	\noindent\textbf{What is a polyform?} A \emph{polyform} is a subset of tiles on a tessellation whose interior is connected. For chess lovers, in the square grid case, this means that a polyform is a subset of tiles that are \emph{rookwise-connected}; that is, a rook can visit every tile of the polyform by moving along rows and columns. In the (Euclidean) plane, there are exactly three regular tessellations built from regular polygons: equilateral triangles, squares, and regular hexagons (Question to you: why only three?). We refer to polyforms based on these tessellations as \emph{polyiamonds}, \emph{polyominoes}, and \emph{polyhexes}, respectively. \\

	\noindent\textbf{What is a fence?} You will be asked to build a larger polyform, which we call a fence, by combining two or more polyforms to enclose an area on the grid. A fence must satisfy the following condition: the tiles not covered by the fence (its complement) must be split into two polyforms (components) that do not share any vertices. Figure~\ref{fig:def_fences} gives an example and a non-example of a fence.  We call \textit{fence configuration} a particular order in which the polyforms are located in a fence, reading clockwise. Observe that, in general, there are multiple fences that have the same fence configuration.	
	
	\begin{challenge} \label{challenge:tetrominoes}
		Take the tetrominoes, depicted in Figure~\ref{fig:tetrapieces}, and use all of them to build a fence that encloses the biggest possible area.
	\end{challenge}

\begin{figure}[h]
	\centering
	\begin{tikzpicture}[scale = \scfa,baseline={(current bounding box.center)}]
	\draw  (0,5) node[left] {\ScissorRight};
	\pic[scale=\scfa] at (1,1) {tetrominoI};
	\draw (1.5,1) node[below] {\polfont i};
	\pic[scale=\scfa] at (4,1) {tetrominoL};
	\draw (4.5,1) node[below] {\polfont l};
	\pic[scale=\scfa] at (7,1) {tetrominoN};
	\draw (8.5,1) node[below] {\polfont n};
	\pic[scale=\scfa] at (11,1) {tetrominoO};
	\draw (12.5,1) node[below] {\polfont o};
	\pic[scale=\scfa] at (14,1) {tetrominoT};
	\draw (15.5,1) node[below] {\polfont t};
	\end{tikzpicture}
	\caption{The five tetrominoes in alphabetical order.}		\label{fig:tetrapieces}
\end{figure}
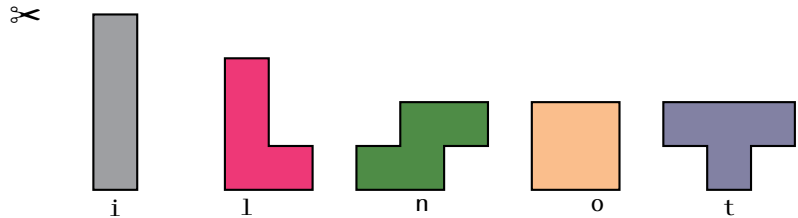

\begin{challenge} \label{challenge:pentominoes}
		There are 12 pentominoes, labelled and depicted in Figure~\ref{pentominoes}. Complete Table~\ref{table:subsetpento} by enclosing the largest possible area using fences built from the indicated subset of pentominoes. We suggest starting with just three pentominoes to gradually build up your fence challenge muscles. Add one pentomino at a time to increase the complexity step by step.
\end{challenge}
	
	\begin{figure}[h]
		\centering
		\begin{tikzpicture}[scale=\scfa]
			\draw  (-1,8) node[left] {\ScissorRight};
			\pic[scale=\scfa] at (0,0) {pentominoF};
			\draw (1.5,0) node[below] {\polfont F};
			\pic[scale=\scfa] at (4,0) {pentominoI};
			\draw (4.5,0) node[below] {\polfont I};
			\pic[scale=\scfa] at (8,0) {pentominoL};
			\draw (8.5,0) node[below] {\polfont L};
			\pic[scale=\scfa] at (12,0) {pentominoN};
			\draw (12.5,0) node[below] {\polfont N};
			\pic[scale=\scfa] at (15,0) {pentominoP};
			\draw (15.5,0) node[below] {\polfont P};
			\pic[scale=\scfa] at (19,0) {pentominoT};
			\draw (20.5,0) node[below] {\polfont T};
			\pic[scale=\scfa] at (0,6) {pentominoU};
			\draw (1.5,10) node[below] {\polfont U};
			\pic[scale=\scfa] at (4,6) {pentominoV};
			\draw (5.5,10) node[below] {\polfont V};
			\pic[scale=\scfa] at (8,6) {pentominoW};
			\draw (9.5,10) node[below] {\polfont W};
			\pic[scale=\scfa] at (12,6) {pentominoX};
			\draw (13.5,10) node[below] {\polfont X};
			\pic[scale=\scfa] at (16,5) {pentominoY};
			\draw (16.5,10) node[below] {\polfont Y};
			\pic[scale=\scfa] at (19,5) {pentominoZ};
			\draw (20.5,10) node[below] {\polfont Z};
		\end{tikzpicture}    
		\caption{All the twelve pentominoes in alphabetical order.}
		\label{pentominoes}
	\end{figure}
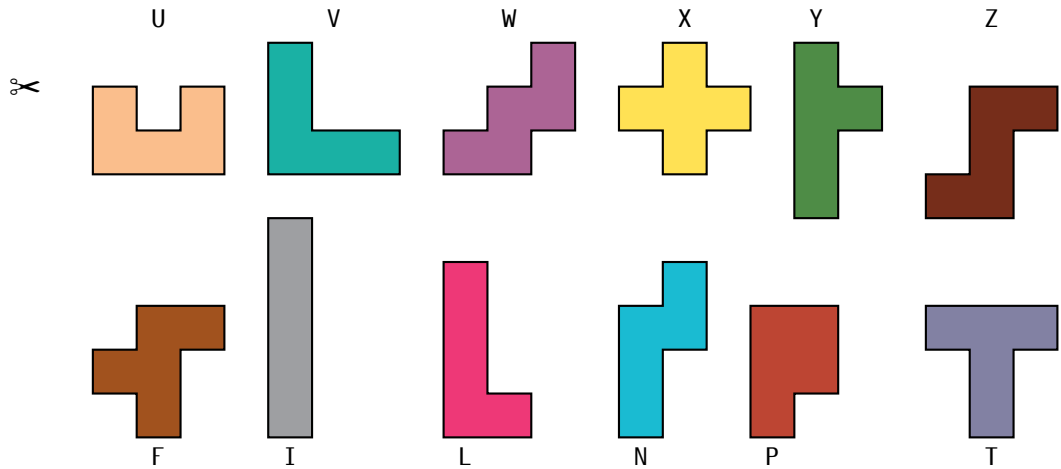
	
	\begin{table}[h]
		\centering
		\begin{tabular}{l|c}
			Pentominoes used & Maximum area\\
			\texttt{F I L} & 4  \\
			\texttt{F I L N} & -- 
			\\
			\texttt{F I L N P} & 19 \\
			\texttt{F I L N P T} & --
			\\
			\texttt{F I L N P T U} & --  
			\\
			\texttt{F I L N P T U V} &--
			\\
			\texttt{F I L N P T U V W} &-- 
			\\
			\texttt{F I L N P T U V W X} &-- 
			\\
			\texttt{F I L N P T U V W X Y} & -- 
			\\
			\texttt{F I L N P T U V W X Y Z} & 128\\
		\end{tabular}
		\caption{Each letter represents one of the 12 pentominoes. 
		}\label{table:subsetpento}
	\end{table}

	\begin{challenge} \label{challenge:polyhexes}
		Polyhexes are polyforms that live in the regular hexagonal tessellation. Take the seven tetrahexes and build, with all of them, a fence that encloses the biggest possible area.
	\end{challenge}

	\begin{figure}[h!]
		\begin{tikzpicture}[hexa/.style= {shape=regular polygon,
				regular polygon sides=6,
				minimum size=\hexscfa cm, draw,
				inner sep=0,anchor=south,
			},scale=\hexscfa]
			\draw  (-1,3) node[left] {\ScissorRight};
			\foreach \j in {0,...,3}{%
				\node[hexa,fill=\couleurhexa] (h0;\j) at ({0},{(\j+1/2)*sin(60)}) {};
			}
			\node[] (h0,-1) at (0,0) {\polfont{bar}};
			\node[hexa,fill=\couleurhexb] (h2;0) at ({2+0},{(0+1/2)*sin(60)}) {};
			\node[hexa,fill=\couleurhexb] (h2;1) at ({2+0},{(1+1/2)*sin(60)}) {};
			\node[hexa,fill=\couleurhexb] (h2;2) at ({2+0},{(2+1/2)*sin(60)}) {};
			\node[hexa,fill=\couleurhexb] (h3;2) at ({2+3/4},{2*sin(60)}) {};
			\node[] (h2,-1) at (2,0) {\polfont{pistol}};
			\node[hexa,fill=\couleurhexc] (h4;0) at ({4+0},{(0+1/2)*sin(60)}) {};
			\node[hexa,fill=\couleurhexc] (h4;1) at ({4+0},{(1+1/2)*sin(60)}) {};
			\node[hexa,fill=\couleurhexc] (h4;2) at ({4+0},{(2+1/2)*sin(60)}) {};
			\node[hexa,fill=\couleurhexc] (h5;3) at ({4+3/4},{3*sin(60)}) {};
			\node[] (h2,-1) at (4,0) {\polfont{worm}};
			\node[hexa,fill=\couleurhexd] (h6;0) at ({6+3/4},{(2)*sin(60)}) {};
			\node[hexa,fill=\couleurhexd] (h6;1) at ({6+0},{(0+1/2)*sin(60)}) {};
			\node[hexa,fill=\couleurhexd] (h6;2) at ({6+0},{(1+1/2)*sin(60)}) {};
			\node[hexa,fill=\couleurhexd] (h6;3) at ({6+3/4},{3*sin(60)}) {};
			\node[] (h2,-1) at (6,0) {\polfont{wave}};
			\node[hexa,fill=\couleurhexe] (h8;0) at ({8+3/4},{1*sin(60)}) {};
			\node[hexa,fill=\couleurhexe] (h8;1) at ({8+0},{(0+1/2)*sin(60)}) {};
			\node[hexa,fill=\couleurhexe] (h8;2) at ({8+0},{(2+1/2)*sin(60)}) {};
			\node[hexa,fill=\couleurhexe] (h8;3) at ({8+3/4},{2*sin(60)}) {};
			\node[] (h2,-1) at (8,0) {\polfont{arc}};
			\node[hexa,fill=\couleurhexf] (h0;0) at ({10+3/4},{1*sin(60)}) {};
			\node[hexa,fill=\couleurhexf] (h0;1) at ({10+0},{(0+1/2)*sin(60)}) {};
			\node[hexa,fill=\couleurhexf] (h0;2) at ({10+0},{(1+1/2)*sin(60)}) {};
			\node[hexa,fill=\couleurhexf] (h6;3) at ({10+3/4},{2*sin(60)}) {};
			\node[] (h2,-1) at (10,0) {\polfont{bee}};
			\node[hexa,fill=\couleurhexg] (h0;0) at ({13-3/4},{2*sin(60)}) {};
			\node[hexa,fill=\couleurhexg] (h0;1) at ({13+0},{(0+1/2)*sin(60)}) {};
			\node[hexa,fill=\couleurhexg] (h0;2) at ({13+0},{(1+1/2)*sin(60)}) {};
			\node[hexa,fill=\couleurhexg] (h6;3) at ({13+3/4},{2*sin(60)}) {};
			\node[] (h2,-1) at (13,0) {\polfont{propeller}};
		\end{tikzpicture}
		\caption{The seven tetrahexes and their common names.}\label{fig:tetrahexes}
	\end{figure}
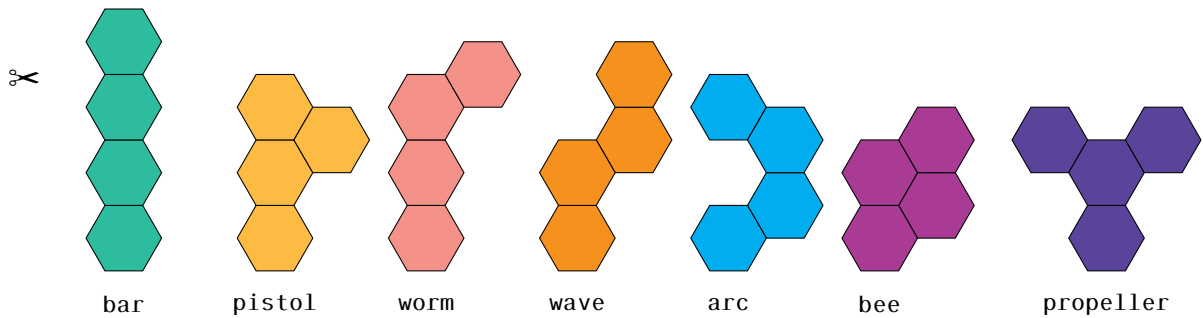

	\begin{challenge} \label{challenge:polyiamonds}
		Polyiamonds are polyforms that live in the regular triangular tessellation. As there are only 3 tetriamonds and 4 pentiamonds (Mini-challenge: find them all just for fun) we start the challenges in the triangular grid with the 12 hexiamonds. Build with all of them a fence that encloses the biggest possible area. 
	\end{challenge}
	
	\begin{challenge}[Proposed by Matthew Kahle, 2024]
		For this challenge, we relax the strict conditions we require of a fence and instead focus on a beautiful extremal topological question: given a polyomino $P$ built from 60 square tiles using the 12 pentominoes, what is the maximum number of \emph{holes}, defined as the number of finite connected components in the complement of $P$, that can be created?
		
		This problem has been completely solved for Euclidean tessellations when it is allowed to use a given number of copies of single polygons (monotile polyforms)~\cite{KR19, MaR, malen2021extremalI, malen2021extremalII}. From these results, upper bounds on the number of holes can be readily established. For instance, a polyomino composed of 60 square tiles can have at most 21 holes~\cite{KR19}; therefore, our polyomino $P$, built from the 12 pentominoes, cannot have more holes than this maximum.
		
		A natural way to generalize the challenge is to consider, for example, 30 copies of the domino, or other polyominoes that all together have 60 squares. Generalizations of these questions can be explored for triangular or hexagonal tessellations, and far beyond into the regular cubical tessellation or the hyperbolic world.
	\end{challenge}
	
		\begin{figure}[h]
		\begin{tikzpicture}[scale=\triscfa]
		\draw  (-2,5) node[left] {\ScissorRight};
		\pic[scale=\triscfa,rotate=60] at (0,2) {hexiaRhomboid};
		\node[] at (0,0) {\polfont{rhomboid}};
		\pic[scale=\triscfa,rotate=60] at (4,2) {hexiaCrook};
		\node[] at (4,0) {\polfont{crook}};
		\pic[scale=\triscfa,rotate=60] at (8,2) {hexiaCrown};
		\node[] at (8,0) {\polfont{crown}};
		\pic[scale=\triscfa,rotate=60] at (12,2) {hexiaSphinx};
		\node[] at (12,0) {\polfont{sphinx}};
		\pic[scale=\triscfa,rotate=60] at (16,2) {hexiaSnake};
		\node[] at (16,0) {\polfont{snake}};
		\pic[scale=\triscfa,rotate=60] at (20,2) {hexiaYacht};
		\node[] at (20,0) {\polfont{yacht}};
		\pic[scale=\triscfa,rotate=60] at (0,-4) {hexiaChevron};
		\node[] at (0,-6) {\polfont{chevron}};
		\pic[scale=\triscfa,rotate=60] at (4,-4) {hexiaSignpost};
		\node[] at (4,-6) {\polfont{signpost}};
		\pic[scale=\triscfa,rotate=60] at (8,-4) {hexiaLobster};
		\node[] at (8,-6) {\polfont{lobster}};
		\pic[scale=\triscfa,rotate=60] at (12,-4) {hexiaShoe};       
		\node[] at (12,-6) {\polfont{shoe}};
		\pic[scale=\triscfa,rotate=60] at (15.5,-4) {hexiaHexagon};
		\node[] at (15.5,-6) {\polfont{hexagon}};
		\pic[scale=\triscfa,rotate=60] at (20.5,-4) {hexiaButterfly};
		\node[] at (20,-6) {\polfont{butterfly}};
		\end{tikzpicture}
		\caption{The 12 hexiamonds and their common names.}\label{fig:hexiamonds}
	\end{figure}
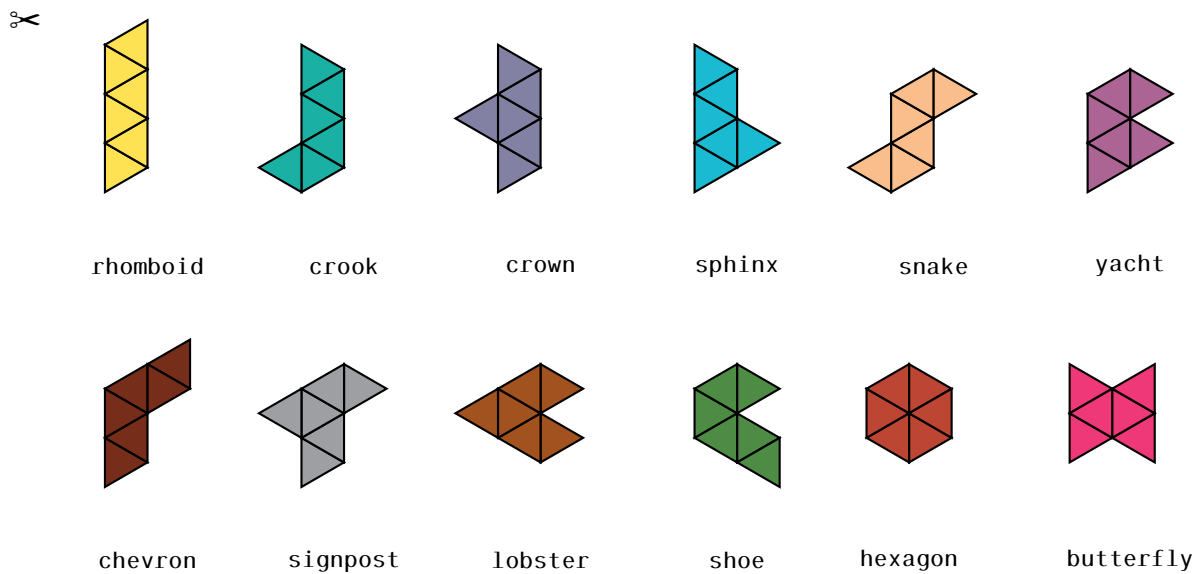
	\begin{remark}
		We encourage the reader to spend some time exploring these challenges and reflecting on them before continuing with the rest of the paper. How would you model these problems mathematically in order to find and prove solutions? Could you design an algorithm to program a computer to solve them? In the remainder of the paper, we return to a more conventional academic style and present some solutions to the initial challenges. But fear not, we also introduce new challenges that might require a bit more mathematical maturity to solve them, guiding you through the mathematics and computer science underlying these problems.

	\end{remark}

	\section{Intro: solutions to Challenges~\ref{challenge:tetrominoes} to~\ref{challenge:polyiamonds} and more challenges}\label{sec:introduction}
	\subsection{Solution to Challenge~\ref{challenge:tetrominoes}} 
	Below we give an example of a fence that encloses an area of 9, which we claim is the maximum area that can be enclosed by a fence using the 5 tetrominoes. 
	One important property of a fence that captures some of its structure is the order in which the polyforms are located along the fence (observe that they form a cycle). Using the notation established in Figure~\ref{fig:tetrapieces}, for instance, the order of the tetrominoes along the fence in Figure~\ref{fig:sol_tetra} can be encoded as \polfont{intol}. Recall that we call \textit{fence configuration} a particular order in which the polyforms are located in a fence. Observe that, in general, there are multiple fences that have the same fence configuration. A natural challenge arises:
	
	\begin{challenge}
		How many tetromino fence configurations can enclose an area of 9? Note that solutions must really differ in the order of the tetrominoes to be counted as different solutions.
	\end{challenge}

	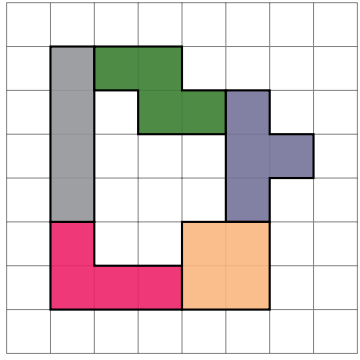
\begin{figure}[h!]
		\centering
		\begin{tikzpicture}[scale=\scfa]
			\draw[help lines] (0,0) grid (8,8);
			\pic[scale = \scfa,rotate=270] at (1,1){tetrominoL={-1}};
			\pic[scale=\scfa,rotate=0] at (1,3) {tetrominoI};
			\pic[scale = \scfa,rotate =0] at (5,5) {tetrominoN={-1}};
			\pic[scale = \scfa,rotate=90] at (7,3){tetrominoT};
			\pic[scale=\scfa,rotate=0] at (4,1) {tetrominoO};
		\end{tikzpicture}
		\caption{One extremal fence of configuration \polfont{intol}  offering a solution to the tetromino fence challenge.}
		\label{fig:sol_tetra}
	\end{figure}

	\subsection{Some history and solution to Challenge~\ref{challenge:pentominoes}}
	
	A set of fence-like puzzles with pentominoes was first presented by Feser in 1968~\cite{Fe68}. One of them,	asking to find the greatest enclosed area by a fence built with the 12 pentominoes, gained popularity when it appeared in 1973 in the famous Gardner’s ``Mathematical Games'' column~\cite{Ga73}. In this column, Gardner presented his best attempt, of area 127 squares, and he challenged the readers to find a better solution. At that time, the maximum area of 128 had already been found by Knuth, who communicated it directly to Gardner. He also had a sketch of proof that it was indeed the maximum solution, but the first complete proof that 128 is, indeed, the maximum was first published by Takakazu Shimauchi in 1978~\cite{Ta78}. 
	
	\begin{figure}[h!]
		\centering
		\begin{tikzpicture}[scale=\scfa]
			\draw[help lines] (0,0) grid (20,20);
			\pic[scale = \scfa,rotate=0] at (2,3){pentominoV};
			\pic[scale=\scfa,rotate=0] at (2,6) {pentominoI};
			\pic[scale = \scfa,rotate =0] at (1,11) {pentominoF};
			\pic[scale = \scfa,rotate=90] at (6,14){pentominoT};   \pic[scale = \scfa,rotate=270] at (6,15){pentominoY={-1}};
			\pic[scale = \scfa,rotate=270] at (10,16) {pentominoP};
			\pic[scale = \scfa,rotate=90] at (16,12) {pentominoW};
			\pic[scale = \scfa,rotate=0] at (14,9) {pentominoX};
			\pic[scale = \scfa,rotate=0] at (16,5) {pentominoL={-1}};
			\pic[scale = \scfa,rotate=270] at (12,2) {pentominoZ={-1}};
			\pic[scale = \scfa,rotate=0] at (9,1) {pentominoU};
			\pic[scale = \scfa,rotate=90] at (9,2) {pentominoN};
		\end{tikzpicture}
		\caption{An extremal pentomino fence enclosing 128 tiles. Its fence configuration is $\polfont{FTYPWXLZUNVI}$.}
		\label{fig:sol_max_pento}
	\end{figure}
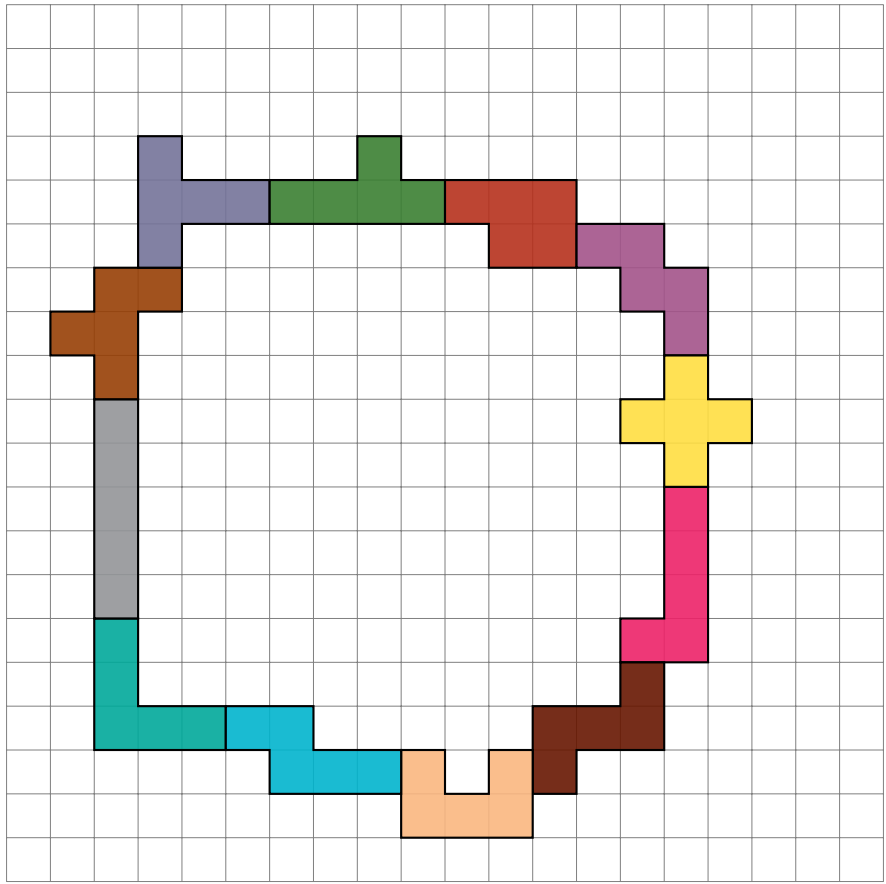

	We have confirmed this result by Shimauchi by an Integer Linear Programming (ILP) implementation and also by going through all the details of his proof \footnote{Thanks to Tamara Sprinkle, we now have an English translation of Shimauchi's paper, which will make it available to a wider community, and we can share it on request.}. The reader can find a version of Shimauchi's proof in Section~\ref{app:proofShimauchi}. We have, furthermore, checked that there are only 1440 fence configurations that can enclose a maximum area of 128 (also mentioned in Gardner's column). 
	
	Other than Shimauchi's publication, we were not able to find any other source proving that 128 is the global maximum for the pentomino fence challenge or any other results related to the challenges that we present in this paper.

	Even though Shimauchi's proof can be adapted to other polyforms in higher dimensions or in non-Euclidean geometries,  without computer assistance, the number of cases to check quickly becomes  unbearable, and even with computer assistance, finding the fence configurations that enclose the global maximum is a hard problem due to the huge number of fence configurations (and the exponential growth of polyforms~\cite{Eden61,klarner1965some}).

	\subsection{Main results}
	We formulated the fence challenges as an Integer Linear Programming (ILP) instance and employed state-of-the-art solvers such as Gurobi~\cite{HH18}. This approach provided a proof of the (global) maximum enclosed area for several subsets of polyforms on the three regular Euclidean  tessellations. Our ILP implementation also revealed something more: it allowed us to compute the number of fence configurations whose maximum enclosed area corresponds to a given value, for the first few larger area sizes. Our main results are synthesised in Theorem~\ref{thm:tab:number_of_solutions}.
	
	\begin{theorem}\label{thm:tab:number_of_solutions}
		The number of fence configurations attaining a given area as their maximum for the three tessellations of the plane are given in the following table. The extremal fence configurations have an area of 128 tiles for the pentominoes, 35 for the tetrahexes and 116 for the hexiamonds.
		\begin{table}[H]
			\begingroup
			\centering
			\begin{tabular}{c|c||c|c||c|c}
				\multicolumn{6}{c}{Extremal fences with polyforms} \\\hline
				\multicolumn{2}{c}{Pentominoes}& \multicolumn{2}{c}{Tetrahexes} & \multicolumn{2}{c}{Hexiamonds}\\\hline
				Area & No. of fence config.  &  Area & No. of fence config. &  Area & No. of fence config. \\\hline\hline
				\textbf{128} & 1440 & \textbf{35} & 2 & \textbf{116} & 546 \\\hline
				127 & 19416 & 34 & 21 & 115 & 2360\\
				126 & 84228 & 33& 89 & & \\
				&& 32 & 101 & & \\
				& & 31 & 115 &&
			\end{tabular}
			\caption{Number of fence configurations having a given area as their maximum. The global maxima for each tessellation are in bold on the first row. ( This table will be updated with all values in the next version.)}
			\endgroup
		\end{table}
	\end{theorem}

	We conclude this section by presenting some extremal fences that solve Challenges~\ref{challenge:polyhexes} and~\ref{challenge:polyiamonds}. In Section~\ref{app:proofShimauchi}, we provide a version of Shimauchi's proof, incorporating improvements and clarifications. Then, in Section~\ref{sec:ilp} we explain the mathematical modelling and algorithm behind our ILP implementation.

	\subsection*{Solution to Challenge~\ref{challenge:polyhexes}}
	There are 7 tetrahexes and in this case we give the only two fence configurations that enclose the maximum area possible amongst all fences with 7 tetrahexes; see Figure~\ref{fig:sols_tetrahex}.
	
	\begin{figure}[h!]
		\centering
		\begin{tikzpicture}[hexa/.style= {shape=regular polygon,
				regular polygon sides=6,
				minimum size=\hexscfa cm, draw,
				inner sep=0,anchor=south,
			},scale=\hexscfa,rotate=90,transform shape]
			
			\foreach \j in {0,...,10}{%
				\ifodd\j 
				\foreach \i in {0,...,6}{\node[hexa] (h\j;\i) at ({\j/2+\j/4},{(\i+1/2)*sin(60)}) 
					{}        
					;}        
				\else
				\foreach \i in {0,...,7}{\node[hexa] (h\j;\i) at ({\j/2+\j/4},{\i*sin(60)}) 
					{} 
					;}
				\fi}         
			\foreach \j/\i in {3/0,1/1}{\node[hexa,fill=\couleurhexa] (h\j;\i) at ({\j/2+\j/4},{(\i+1/2)*sin(60)}) 
				{}
				;}
			\foreach \j/\i in {4/0,2/1}{\node[hexa,fill=\couleurhexa] (h\j;\i) at ({\j/2+\j/4},{\i*sin(60)}) 
				{}
				;}
			\foreach \j/\i in {1/5,3/6}{\node[hexa,fill=\couleurhexb] (h\j;\i) at ({\j/2+\j/4},{(\i+1/2)*sin(60)}) 
				{}
				;}
			\foreach \j/\i in {2/7,2/6}{\node[hexa,fill=\couleurhexb] (h\j;\i) at ({\j/2+\j/4},{\i*sin(60)}) {};}
			\foreach \j/\i in {9/5,9/4,9/3}{\node[hexa,fill=\couleurhexc] (h\j;\i) at ({\j/2+\j/4},{(\i+1/2)*sin(60)}) 
				{};}
			\foreach \j/\i in {8/6}{\node[hexa,fill=\couleurhexc] (h\j;\i) at ({\j/2+\j/4},{\i*sin(60)}) {};}
			\foreach \j/\i in {5/6,7/6}{\node[hexa,fill=\couleurhexd] (h\j;\i) at ({\j/2+\j/4},{(\i+1/2)*sin(60)}) 
				{};}
			\foreach \j/\i in {6/7,4/7}{\node[hexa,fill=\couleurhexd] (h\j;\i) at ({\j/2+\j/4},{\i*sin(60)}) {};}
			\foreach \j/\i in {1/4,1/2}{\node[hexa,fill=\couleurhexe] (h\j;\i) at ({\j/2+\j/4},{(\i+1/2)*sin(60)}) 
				{};}
			\foreach \j/\i in {0/4,0/3}{\node[hexa,fill=\couleurhexe] (h\j;\i) at ({\j/2+\j/4},{\i*sin(60)}) {};}
			\foreach \j/\i in {5/0,7/0}{\node[hexa,fill=\couleurhexf] (h\j;\i) at ({\j/2+\j/4},{(\i+1/2)*sin(60)}) 
				{};}
			\foreach \j/\i in {6/1,6/0}{\node[hexa,fill=\couleurhexf] (h\j;\i) at ({\j/2+\j/4},{\i*sin(60)}) {};}
			\foreach \j/\i in {9/1,9/2}{\node[hexa,fill=\couleurhexg] (h\j;\i) at ({\j/2+\j/4},{(\i+1/2)*sin(60)}) 
				{};}
			\foreach \j/\i in {8/1,10/1}{\node[hexa,fill=\couleurhexg] (h\j;\i) at ({\j/2+\j/4},{\i*sin(60)}) {};}
		\end{tikzpicture}
		\quad
		\begin{tikzpicture}[hexa/.style= {shape=regular polygon,
				regular polygon sides=6,
				minimum size=\hexscfa cm, draw,
				inner sep=0,anchor=south,
			},scale=\hexscfa]
			
			\foreach \j in {0,...,9}{%
				\ifodd\j 
				\foreach \i in {0,...,8}{\node[hexa] (h\j;\i) at ({\j/2+\j/4},{(\i+1/2)*sin(60)}) 
					{}        
					;}        
				\else
				\foreach \i in {0,...,9}{\node[hexa] (h\j;\i) at ({\j/2+\j/4},{\i*sin(60)}) 
					{} 
					;}
				\fi}         
			\foreach \j/\i in {1/6,3/7}{\node[hexa,fill=\couleurhexa] (h\j;\i) at ({\j/2+\j/4},{(\i+1/2)*sin(60)}) 
				{}
				;}
			\foreach \j/\i in {4/8,2/7}{\node[hexa,fill=\couleurhexa] (h\j;\i) at ({\j/2+\j/4},{\i*sin(60)}) 
				{}
				;}
			\foreach \j/\i in {1/5,1/4,1/3}{\node[hexa,fill=\couleurhexb] (h\j;\i) at ({\j/2+\j/4},{(\i+1/2)*sin(60)}) 
				{}
				;}
			\foreach \j/\i in {0/5}{\node[hexa,fill=\couleurhexb] (h\j;\i) at ({\j/2+\j/4},{\i*sin(60)}) {};}
			\foreach \j/\i in {9/5,9/4,9/6}{\node[hexa,fill=\couleurhexc] (h\j;\i) at ({\j/2+\j/4},{(\i+1/2)*sin(60)}) 
				{};}
			\foreach \j/\i in {8/7}{\node[hexa,fill=\couleurhexc] (h\j;\i) at ({\j/2+\j/4},{\i*sin(60)}) {};}
			\foreach \j/\i in {9/3,7/1}{\node[hexa,fill=\couleurhexd] (h\j;\i) at ({\j/2+\j/4},{(\i+1/2)*sin(60)}) 
				{};}
			\foreach \j/\i in {8/2,8/3}{\node[hexa,fill=\couleurhexd] (h\j;\i) at ({\j/2+\j/4},{\i*sin(60)}) {};}
			\foreach \j/\i in {5/8,7/8,7/7}{\node[hexa,fill=\couleurhexe] (h\j;\i) at ({\j/2+\j/4},{(\i+1/2)*sin(60)}) 
				{};}
			\foreach \j/\i in {6/9}{\node[hexa,fill=\couleurhexe] (h\j;\i) at ({\j/2+\j/4},{\i*sin(60)}) {};}
			\foreach \j/\i in {5/0,5/1}{\node[hexa,fill=\couleurhexf] (h\j;\i) at ({\j/2+\j/4},{(\i+1/2)*sin(60)}) 
				{};}
			\foreach \j/\i in {6/1,4/1}{\node[hexa,fill=\couleurhexf] (h\j;\i) at ({\j/2+\j/4},{\i*sin(60)}) {};}
			\foreach \j/\i in {3/1,1/1}{\node[hexa,fill=\couleurhexg] (h\j;\i) at ({\j/2+\j/4},{(\i+1/2)*sin(60)}) 
				{};}
			\foreach \j/\i in {2/2,2/3}{\node[hexa,fill=\couleurhexg] (h\j;\i) at ({\j/2+\j/4},{\i*sin(60)}) {};}
		\end{tikzpicture}
		\caption{Two tetrahex fences realising the two extremal fence configurations for the tetrahex fence challenge. The (global) maximum area that they enclose consists of 35 hexagons.}
		\label{fig:sols_tetrahex}
	\end{figure}
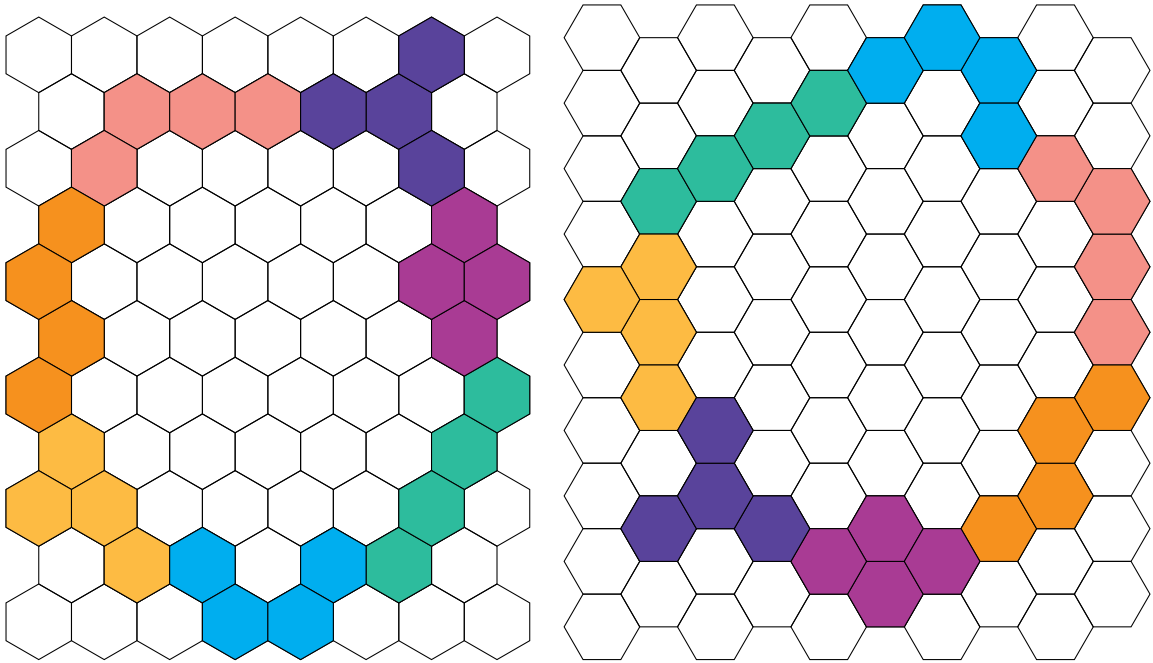

	\subsection*{Solution to Challenge~\ref{challenge:polyiamonds}}
	
	As a first step, we focus on pentiamonds. Since there are only four pentiamonds, this can be achieved in a short amount of time, and it allows us to provide some geometric insights before presenting the solution for the hexiamonds. An extremal fence built using the 4 pentiamonds is shown in Figure~\ref{fig:sol_pentiamonds}.
	\begin{figure}[h]
		\begin{tikzpicture}[scale=\triscfa]
			\newcommand{\nbx}{7}
			\newcommand{\nby}{7}
			
			
			\clip (tri cs: x=2,y=0) -- (tri cs:x=5,y=0) -- (tri cs:x=5,y=3) -- (tri cs: x=2,y=6) -- (tri cs: x=-1,y=6) -- (tri cs: x=-1,y=3) -- cycle; 
			\foreach \j in {-1,...,\the\numexpr\nby-1}
			{\foreach \i in {-1,...,\the\numexpr\nbx-1} 
				{
			}}
			\foreach \j in {-1,...,\the\numexpr\nby-1}
			{\foreach \i in {-1,...,\the\numexpr\nbx-1} 
				{
					\draw (tri cs: x=\i,y=\j) -- (tri cs:x=\i,y=\nby-1);
					\draw (tri cs: x=\i,y=\j) -- (tri cs:x=\nbx-1,y=\j);
					\draw (tri cs: x=\i,y=\j) -- (tri cs:x=\j,y=\i);
			}    }
			\pic[scale=\triscfa,rotate=0] at (tri cs:x=4,y=1) {pentiaBar={-1}};;
			\pic[scale=\triscfa,rotate=120] at (tri cs:x=2,y=5) {pentiaSphinx={1}};
			\pic[scale=\triscfa,rotate=300] at (tri cs:x=2,y=3) {pentiaHexa};
			\pic[scale=\triscfa,rotate=240] at (tri cs:x=0,y=3) {pentiaCrook={-1}}; 
		\end{tikzpicture}
		\caption{One of the extremal pentiamond fences attaining an area of 5.}\label{fig:sol_pentiamonds}
	\end{figure}
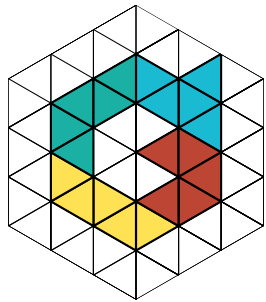
	
	Observing the \textit{shape} of this extremal fence, we would like to offer a hint for the hexiamond fence challenge, in case you have not yet found an extremal configuration.
	
	\begin{hint}
		The extremal fences built with hexiamonds tend to resemble a hexagon. This means there will be six ``corner'' pieces. The key idea is to place the \textit{best} hexiamonds in those corners. For example, \polfont{chevron} could be an excellent corner piece, as are \polfont{crook}, \polfont{snake}, and \polfont{signpost}. Why?
	\end{hint}
	
	These observations are based in the \emph{isoperimetric properties of polyforms}, which vary depending on the tessellation, and have been shown to play a key role in extremal topological problems~\cite{harary1976extremal, malen2021extremalI, malen2021extremalII, beautifulanimals, beautifulanimals2}. We will not hold the suspense much longer: an extremal fence configuration for the hexiamonds is presented in Figure~\ref{fig:sols_hexiamonds}.
	
	
	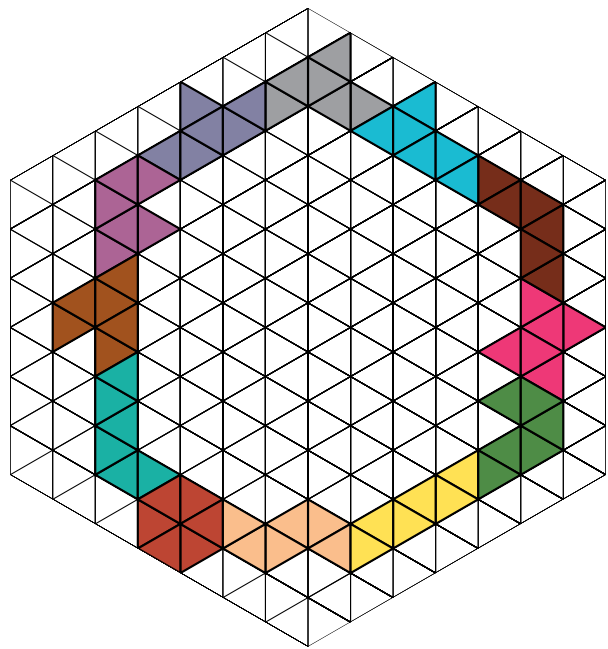
\begin{figure}[H]
		\begin{tikzpicture}[scale=\triscfa]
			\newcommand{\nbx}{14}
			\newcommand{\nby}{14}
			
			\clip (tri cs: x=6,y=0) -- (tri cs:x=13,y=0) -- (tri cs:x=13,y=7) -- (tri cs: x=7,y=13) -- (tri cs: x=0,y=13) -- (tri cs: x=0,y=6) -- cycle; 
			
			\foreach \j in {0,...,\the\numexpr\nby-1}
			{\foreach \i in {0,...,\the\numexpr\nbx-1} 
				{
			}}
			\foreach \j in {0,...,\the\numexpr\nby-1}
			{\foreach \i in {0,...,\the\numexpr\nbx-1} 
				{
					\draw (tri cs: x=\i,y=\j) -- (tri cs:x=\i,y=\nby-1);
					\draw (tri cs: x=\i,y=\j) -- (tri cs:x=\nbx-1,y=\j);
					\draw (tri cs: x=\i,y=\j) -- (tri cs:x=\j,y=\i);
			}    }
			\pic[scale=\triscfa,rotate=0] at (tri cs:x=3,y=11) {hexiaSignpost={-1}};
			\pic[scale=\triscfa,rotate=120] at (tri cs:x=10,y=10) {hexiaButterfly};
			\pic[scale=\triscfa,rotate=300] at (tri cs:x=6,y=11) {hexiaChevron};
			\pic[scale=\triscfa,rotate=180] at (tri cs:x=3,y=12) {hexiaSphinx={-1}};
			\pic[scale=\triscfa,rotate=0] at (tri cs:x=2,y=8) {hexiaCrown};
			\pic[scale=\triscfa,rotate=60] at (tri cs:x=3,y=6) {hexiaYacht};
			\pic[scale=\triscfa,rotate=300] at (tri cs:x=3,y=4) {hexiaLobster};
			\pic[scale=\triscfa,rotate=300] at (tri cs:x=7,y=1) {hexiaCrook={-1}};
			\pic[scale=\triscfa,rotate=300] at (tri cs:x=8,y=1) {hexiaHexagon};
			\pic[scale=\triscfa,rotate=180] at (tri cs:x=10,y=3) {hexiaSnake={1}};\pic[scale=\triscfa,rotate=60] at (tri cs:x=12,y=5) {hexiaRhomboid={-1}};
			\pic[scale=\triscfa,rotate=300] at (tri cs:x=12,y=6) {hexiaShoe={-1}};  
		\end{tikzpicture}
		\caption{An extremal hexiamond fence enclosing an area of 116 triangles.}
		\label{fig:sols_hexiamonds}
	\end{figure}
	
	\section{Proof of the pentomino fence à la Shimauchi}\label{app:proofShimauchi}
	
	We now prove that the maximum area that can be enclosed by a pentomino fence configuration is at most 128 tiles. We already know that this maximum is indeed attained by certain fence configurations, for instance, the one depicted in Figure~\ref{fig:sol_max_pento}. As mentioned earlier, we identified all extremal fence configurations using the ILP approach described in Section~\ref{sec:ilp}.
	
	The proof we provide here is an English translation, with some improvements, of one originally published by Shimauchi~\cite{Ta78} in 1978. To our knowledge, this is the first English version of this proof.
	
	Readers may first wish to consult the proof we presented for the tetromino challenge in~\cite{LRRR24}. That proof can be used to introduce mathematical modelling techniques to students, who can then apply the same ideas to tackle the pentomino fence challenge.

	\begin{theorem}\label{thm:maxpento}
		The extremal fence configurations that solve the pentomino fence challenge enclose an area of 128 tiles.
	\end{theorem}
	
	We prove the theorem with a series of lemmata. Before starting, however, we will assign certain geometric quantities to each pentomino. This is the main idea of Shimauchi.  He calls them \emph{progression}, \emph{protrusion} and \emph{length}. The progression is the distance it moves along its initial direction before bending. The protrusion is how  much it moves after bending. The length is then the sum of the progression and protrusion, plus one for the initial tile; this represents the contribution of the pentomino to the length of the fence.

	The pentominoes are grouped into families according to their geometric impact on the fence. The main property of these families is that pentominoes belonging to the same family can be switched without changing the size of the fence. This can be visualised, for instance, by switching the \begin{tikzpicture}[baseline={(current bounding box.center)}]
		\pic[scale=.15,rotate=0,xscale=-1] at (0,0) {pentominoL={-1}};
	\end{tikzpicture} ($\polfont{L}$) and \begin{tikzpicture}[baseline={(current bounding box.center)}]
		\pic[scale=.15,rotate=90] at (0,0) {pentominoN};
	\end{tikzpicture} ($\polfont N$) pentominoes in Figure~\ref{fig:sol_max_pento}.
	
	\begin{enumerate}
		\item Pentomino: $\{\polfont I\}$; progression: 4, protrusion: 0, length: 5.
		
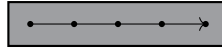
\[
		\begin{tikzpicture}[scale = \scfa]
			\pic[scale=\scfa,rotate=90] at (0,0) {pentominoI};
			\foreach \x/\y in {-5/0,-4/0,-3/0,-2/0,-1/0}
			{
				\fill (\x+.5,\y+.5) circle(2pt);
			}
			\draw[->] (-4.5,.5) -- (-.5,.5);
		\end{tikzpicture}
		\] 
		\item  Pentominoes: $\{ \polfont{L,N}\}$; progression: 3, protrusion: 1, length: 5 ($\polfont{N}$ has one defect---indicated by ``\begin{tikzpicture}[scale=\scfa,baseline={(current bounding box.center)}]
			\draw[pattern=north west lines, pattern color=red] (-.5,-.5) rectangle (.5,.5);
		\end{tikzpicture}''--- inside the fence).
		
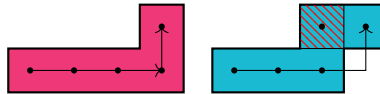
\[
		\begin{tikzpicture}[scale = \scfa]
			\pic[scale=\scfa,rotate=90] at (0,0) {pentominoL};
			\foreach \x/\y in {-4/0,-3/0,-2/0,-1/0,-1/1}
			{
				\fill (\x+.5,\y+.5) circle(2pt);
			}
			\draw[->] (-3.5,0.5) -- (-.5,.5);
			\draw[->] (-.5,.5) -- (-.5,1.5);
		\end{tikzpicture}
		\quad 
		\begin{tikzpicture}[scale = \scfa]
			\pic[scale=\scfa,rotate=270] at (0,0) {pentominoN={-1}};
			\draw[pattern=north west lines, pattern color=red] (2,1) rectangle (3,2);
			\foreach \x/\y in {0/0,2/0,1/0,2/1,3/1}
			{
				\fill (\x+.5,\y+.5) circle(2pt);
			}
			
			\draw[->] (0.5,0.5) -- (3.5,0.5) -- (3.5,1.5);
		\end{tikzpicture}
		\] 
		\item Pentominoes: $\{\polfont{V,W,Z}\}$; progression: 2, protrusion: 2, length:  5  ($\polfont W$ has one defect and $\polfont{Z}$ has two defects---indicated by ``\begin{tikzpicture}[scale=\scfa,baseline={(current bounding box.center)}]
			\draw[pattern=north west lines, pattern color=red] (-.5,-.5) rectangle (.5,.5);
		\end{tikzpicture}''--- inside the fence).
		
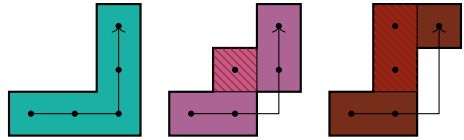
\[
		\begin{tikzpicture}[scale = \scfa]
			\pic[scale=\scfa,rotate=90] at (0,0) {pentominoV};
			\foreach \x/\y in {-1/0,-1/1,-1/2,-2/0,-3/0}
			{
				\fill (\x+.5,\y+.5) circle(2pt);
			}
			\draw[->] (-2.5,0.5) -- (-.5,.5)-- (-.5,2.5);
		\end{tikzpicture}
		\quad 
		\begin{tikzpicture}[scale = \scfa]
			\pic[scale=\scfa] at (0,0) {pentominoW};
			\draw[pattern=north west lines, pattern color=red] (1,1) rectangle (2,2);
			\foreach \x/\y in {0/0,1/0,1/1,2/1,2/2}
			{
				\fill (\x+.5,\y+.5) circle(2pt);
			}
			\draw[->] (.5,0.5) -- (2.5,.5) -- (2.5,2.5);
		\end{tikzpicture}
		\quad 
		\begin{tikzpicture}[scale = \scfa]
			\pic[scale=\scfa] at (0,0) {pentominoZ};
			\draw[pattern=north west lines, pattern color=red] (1,1) rectangle (2,3);
			\foreach \x/\y in {0/0,1/0,1/1,1/2,2/2}
			{
				\fill (\x+.5,\y+.5) circle(2pt);
			}
			\draw[->] (.5,0.5) -- (2.5,.5)-- (2.5,2.5);
		\end{tikzpicture}
		\] 
		\item Pentomino: $\{\polfont Y\}$; progression: 3, protrusion: 0, length: 4.
		
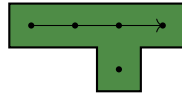
\[
		\begin{tikzpicture}[scale = \scfa]   
			\pic[scale=\scfa,rotate=270] at (0,0) {pentominoY};
			\foreach \x/\y in {0/-1,1/-1,2/-1,2/-2,3/-1}
			{
				\fill (\x+.5,\y+.5) circle(2pt);
			}
			\draw[->] (0.5,-.5) -- (3.5,-.5);
		\end{tikzpicture}
		\] 
		\item Pentominoes: $\{\polfont{T,F,P}\}$; progression: 1, protrusion: 2, length: 4  ($\polfont{P}$ has one defect---indicated by ``\begin{tikzpicture}[scale=\scfa,baseline={(current bounding box.center)}]
			\draw[pattern=north west lines, pattern color=red] (-.5,-.5) rectangle (.5,.5);
		\end{tikzpicture}''--- inside the fence).
		
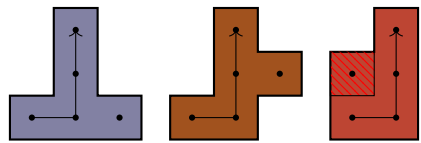
\[
		\begin{tikzpicture}[scale = \scfa]
			\pic[scale=\scfa,rotate=180] at (0,0) {pentominoT};
			\foreach \x/\y in {-3/-3,-2/-3,-2/-1,-2/-2,-1/-3}
			{
				\fill (\x+.5,\y+.5) circle(2pt);
			}
			\draw[->] (-2.5,-2.5) -- (-1.5,-2.5) -- (-1.5,-.5);
		\end{tikzpicture} \quad 
		\begin{tikzpicture}[scale = \scfa]
			\pic[scale=\scfa,rotate=180] at (0,0) {pentominoF};
			\foreach \x/\y in {-3/-3,-2/-3,-2/-2,-2/-1,-1/-2}
			{
				\fill (\x+.5,\y+.5) circle(2pt);
			}
			\draw[->] (-2.5,-2.5) -- (-1.5,-2.5) -- (-1.5,-.5);
		\end{tikzpicture} \quad 
		\begin{tikzpicture}[scale = \scfa]
			\pic[scale=\scfa,rotate=180] at (0,0) {pentominoP};
			\draw[pattern=north west lines, pattern color=red] (-2,-2) rectangle (-1,-1);
			\foreach \x/\y in {-2/-3,-2/-2,-1/-1,-1/-2,-1/-3}
			{
				\fill (\x+.5,\y+.5) circle(2pt);
			}
			\draw[->] (-1.5,-2.5) -- (-.5,-2.5) -- (-.5,-.5);
		\end{tikzpicture}
		\] 
		For these, we can also look at them with progression 2 and protrusion 1 by looking at them  rotated $90^\circ$.
		\item Pentominoes: $\{\polfont{X,U}\}$; progression: 2, protrusion: 0, length: 3 ($\polfont{X}$ has one defect---indicated by ``\begin{tikzpicture}[scale=\scfa,baseline={(current bounding box.center)}]
			\draw[pattern=north west lines, pattern color=red] (-.5,-.5) rectangle (.5,.5);
		\end{tikzpicture}''---inside the fence, and $\polfont{U}$ has a dent---denoted ``\begin{tikzpicture}[scale=\scfa,baseline={(current bounding box.center)}]
			\draw[pattern=north east lines, pattern color=green] (-.5,-.5) rectangle (.5,.5);
		\end{tikzpicture}''--- leaving an extra place inside the fence). 
		
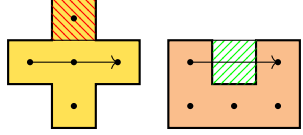
\[
		\begin{tikzpicture}[scale = \scfa]
			\pic[scale=\scfa] at (0,0) {pentominoX};
			\draw[pattern=north west lines, pattern color=red] (1,2) rectangle (2,3);
			\foreach \x/\y in {0/1,1/0,1/1,1/2,2/1}
			{
				\fill (\x+.5,\y+.5) circle(2pt);
			}
			\draw[->] (0.5,1.5) -- (2.5,1.5);
		\end{tikzpicture}
		\quad 
		\begin{tikzpicture}[scale = \scfa]
			\pic[scale=\scfa] at (0,0) {pentominoU};
			\draw[pattern=north east lines, pattern color=green] (1,1) rectangle (2,2);
			\foreach \x/\y in {0/0,1/0,0/1,2/0,2/1}
			{
				\fill (\x+.5,\y+.5) circle(2pt);
			}
			\draw[->] (0.5,1.5) -- (2.5,1.5);
		\end{tikzpicture}
		\]
		
	\end{enumerate}
	The pentomino $\polfont U$ could also be placed in the family (5) with $\{\polfont{T,F,P}\}$, where it would have length 4. The price to pay is a defect. One advantage of placing it in the family (6) is that it gains a dent. The following lemma will make this choice clear, as it will show that we can use one polyomino suboptimally and reach the maximum perimeter.

	\begin{lemma}\label{lem:perimeter_pentomino}
		The perimeter of a fence using the 12 pentominoes is bounded above by $53$.
	\end{lemma}
	\begin{proof}
		Taking the maximum contribution to the length of all the pieces, so in particular taking $\polfont{U}$ to have length 4, we get $ 5\times |\{\polfont{I,L,N,V,W,Z}\}|+4\times|\{\polfont{Y,T,F,P,U}\}| + 3\times|\{\polfont{X}\}| = 53$.
	\end{proof}
	As 53 is odd, it cannot be the perimeter of a rectangle; hence, we can take 52 as the maximum. In practice, this will be done by placing $\polfont{U}$ in the same category as $\polfont X$, thus using it to add $3$ to the length. This choice of $\polfont{U}$ is made because it then gains a \textit{dent}, whereas if we used it to have progression 2 and protrusion 1 (so as the family (5) with ``L''-shape), then we would have a defect inside.

	\begin{lemma}\label{lem:maxsol_square}
		There is an extremal pentomino fence with an inner area of 128. 
	\end{lemma}
	\begin{proof}
		Figure~\ref{fig:sol_max_pento} gives a solution.
	\end{proof}
	
	We prove a corollary of the two previous results, which we used in optimizing the ILP formulation in Section~\ref{sec:ilp}.
	
	\begin{corollary}\label{coro:sizeboardpentomino}
		Every extremal pentomino fence fits within a $20\times 20$ square.
	\end{corollary}
	\begin{proof}
		Suppose that we have an extremal fence that does not fit within a $20\times 20$ square. Then the diameter of the fence must be more than 21 tiles long. Then it is either inscribed in a rectangle or in a rhombus.
		
		If it is inscribed in a rectangle, with the restriction of the perimeter of Lemma~\ref{lem:perimeter_pentomino}, then the biggest rectangle would have to be of size $21\times 6$, and the area (126) is below the solution found in Lemma~\ref{lem:maxsol_square}. 
		
		If it is embedded within a rhombus whose big diagonal has a length of 21, then to respect the perimeter constraints of Lemma~\ref{lem:perimeter_pentomino}, the rhombus would need to have a small diagonal of maximum size 16, so that the perimeter of the rhombus is 52. But this is not achievable with the pentominoes, because this would mean a side distance of 13 units diagonally, which requires approximately twice the number of tiles (for example $\polfont{I,J,N,V,Y}$ can make such a movement at the cost of 25 tiles), but even $21\times 12$ rhombi have an area of only $21\times 12/2=126$, which is not sufficient. 
		
		Thus any extremal fence will fit in a $20­\times 20$ square.
	\end{proof}

	\begin{remark}
		Any extremal fence must use the optimal geometric properties of each pentomino. This introduces defects (from $\polfont{N,W,Z,P,X}$) inside the fence that unavoidably reduce the inner area, precisely by 6. On the other hand, $\polfont{U}$ has a dent, so it will allow for one extra square in the inner area.
	\end{remark}

	\begin{figure}[h!]
		\centering
		\begin{tikzpicture}[scale=\scfa]
			\draw[help lines] (1,1) grid (17,17);
			\pic[scale = \scfa,rotate=0] at (2,3){pentominoV};
			\pic[scale=\scfa,rotate=0] at (2,6) {pentominoI};
			\pic[scale = \scfa,rotate =0] at (1,11) {pentominoF};
			\pic[scale = \scfa,rotate=90] at (6,14){pentominoT};   \pic[scale = \scfa,rotate=270] at (6,15){pentominoY={-1}};
			\pic[scale = \scfa,rotate=270] at (10,16) {pentominoP};
			\pic[scale = \scfa,rotate=90] at (16,12) {pentominoW};
			\pic[scale = \scfa,rotate=0] at (14,9) {pentominoX};
			\pic[scale = \scfa,rotate=0] at (16,5) {pentominoL={-1}};
			\pic[scale = \scfa,rotate=270] at (12,2) {pentominoZ={-1}};
			\pic[scale = \scfa,rotate=0] at (9,1) {pentominoU};
			\pic[scale = \scfa,rotate=90] at (9,2) {pentominoN};
			\foreach \x/\y in {2/2,3/2,4/2,2/14,2/15,15/2,15/3,15/4,15/15,14/15,13/15}
			{
				\fill[red] (\x+.5,\y+.5) circle(2pt);
			}
			
			\foreach \x/\y in {6/3,12/3,13/3,14/10,14/13,11/14}
			{
				\draw[pattern=north west lines, pattern color=red] (\x,\y) rectangle (\x+1,\y+1);
			}
			\draw[pattern=north east lines, pattern color=green] (10,2) rectangle (11,2+1);
			\draw[thick, dotted,->] (2+.5,3+.5)--(5+.5,3+.5);
			\draw[thick, dotted,->](5+.5,3+.5)--(5+.5,2+.5)--(14+.5,2+.5);
			\draw[thick, dotted,->](14+.5,2+.5)--(14+.5,5+.5)--(15+.5,5+.5);
			\draw[thick, dotted,->](15+.5,5+.5)--(15+.5,14+.5)--(12+.5,14+.5);
			\draw[thick, dotted,->](12+.5,14+.5)--(12+.5,15+.5)--(3+.5,15+.5);
			\draw[thick, dotted,->](3+.5,15+.5)--(3+.5,13+.5)--(2+.5,13+.5)--(2+.5,3+.5);
		\end{tikzpicture}
		\caption{Example of one of the extremal pentomino fences 
			 of 128 tiles with corner losses, defects and dents indicated, and the length drawn.}
		\label{fig:sol_max_pento_comm}
	\end{figure}
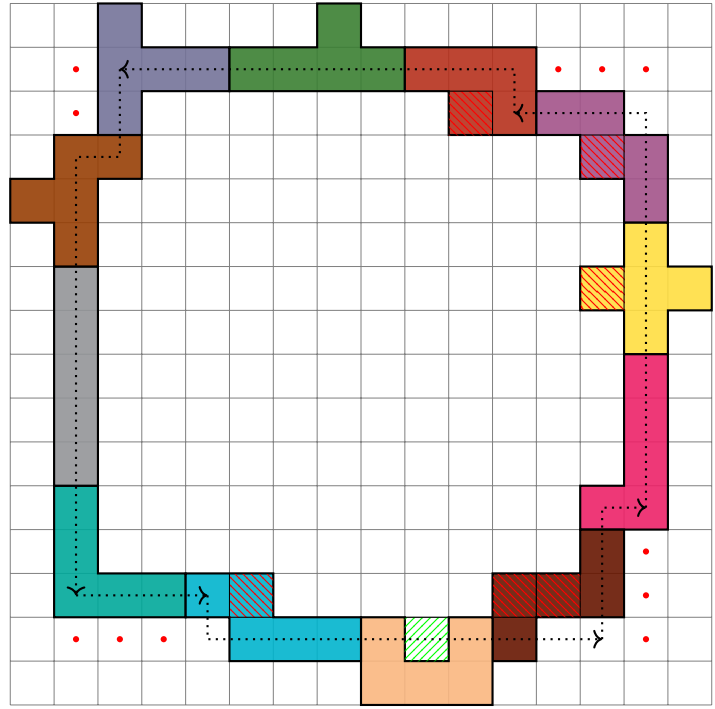

	\begin{lemma}\label{lem:rectangle_inscribing_pento}
		Extremal pentomino fences must have perimeter $52$ and are inscribed in one of the following rectangles (with respective inner areas): $13\times 13$ (144), $14\times 12$ (143), $15\times 11$ (140), and $16\times 10$ (135).
	\end{lemma}
	\begin{proof}
		An extremal fence must have at least an area of $128$ by Lemma~\ref{lem:maxsol_square}. The area of a $17\times 9$ rectangle is 128, but the pentominoes unavoidably lower the area as remarked above, therefore, this rectangle is disqualified. Furthermore, note that if the perimeter is 50, then the maximum inner rectangle is $13\times 12$, which has an inner area of $132$, and the pentomino fence would also diminish it below 128, as remarked above. 
	\end{proof}
	
We are now ready for the proof of the main theorem of the section.
	\begin{proof}[Proof of Theorem~\ref{thm:maxpento}]
		Let us assume we have a solution inscribed in a $13\times 13$ square. The question becomes where to place the pieces so that we approach the maximum area. The heuristic is to place the straightest pieces on the sides and the bendiest in the corners.  Hence, the pieces of family (1), (4), and (6) should be placed on the side to maximize. It is also clear that the three pieces of family (3) are better placed on the corner, since they are the most corner-like. Then one corner remains from the remaining families (2) or (5). (In Figure~\ref{fig:sol_max_pento}, we see that $``\polfont{T}''$ is the chosen one.)

		Precisely, this choice means that we will depart from 11 (3 for each of the corners of the family (3) and 2 for the corner of the family (2) or (5)) from the inner area of the square. Then, we can compute that the maximum area will be $144-11-6+1 = 128$ and shall happen in the solutions inscribed in a $13\times 13$ square, as the arguments with the corners still hold. 
		
		Now, for all other possibilities of Lemma~\ref{lem:rectangle_inscribing_pento}, we compute the maximum area, keeping in mind that we must diminish by 11 for the corner losses, 6 for the defects, and add one for the dent. Then we see that the solution must be in the $13\times 13$ as the second possibility, the $14\times 12$ rectangle, has an inner area of 143, but $143-11-6+1 = 127<128$.
	\end{proof}

	\begin{challenge}\label{challenge:proof_tetrahexes}
		Adapt Shimauchi's proof to the tetrahex case. That is, prove that an extremal fence for the tetrahex fence challenge encloses an area of 35 tiles.
	\end{challenge}
	
	\begin{challenge}\label{challenge:proof_hexiamonds}
		Adapt Shimauchi's proof to fences built with the 12 hexiamonds. In particular, prove that an extremal fence in the hexiamond fence challenge encloses an area of 116 tiles.
	\end{challenge}

	\section{Modelling with Integer Linear Programming}\label{sec:ilp}
	
	In this section, we present the integer linear programming model that we have implemented\footnote{A similar model for the  tetromino fence challenge without allowing flipping was done by D.W.~\cite{DWstack}. The model presented here is more general, incorporating other tessellations, flipping and higher dimensions.} for exploring and finding solutions to the fence challenges. 
	
	First, we choose a dimension $d\geq 2$ and a tessellation $T$ of the $d$-dimensional space. We denote $P_T^m$ the number of polyforms of size $m$ in the tessellation T. For example, $P_\square^5 = 12$ for the square tessellation of the plane. 
	
	Then, we fix a finite region $R$ of the space in which the program will search for potential fences. Choosing the best region is a delicate task: if $R$ is too big, it will slow the program, but if it is too small, we run the risk of missing optimal solutions.  For example, a good choice of $R$ for the pentomino fence challenge ($d=2$ and $T = T_{\square}$ and $m=5$ would be a square of inner dimensions $20\times 20$ by Corollary~\ref{coro:sizeboardpentomino}---technically then the total region in the implementation would be a $22\times 22$ square due to the interior constraints). Since there are finitely-many tiles in $R$, we can enumerate them. This means that every tile $l\in R$ will be labelled by its number $l$ in the enumeration.
	
	Now, the variables necessary for our model can be defined. Given a fence $F$, for every tile $l\in R$ we define the binary variable $x_l$, which is true if and only if $l$ is part of the interior of the fence. For every tile $l$ we define the binary variable $y_l$, which is true if and only if $l$ is part of the fence $F$ itself, meaning that it is covered by a polyform. 
	
	Let $S_T$ be the set of possible symmetries of a polyform in a tessellation $T$ and $s_T:=|S_T|$. For example, $s_T=8$ for the square tessellation of the plane, and  we give all the potential symmetries of the \begin{tikzpicture}[baseline={(current bounding box.center)}]
		\pic[scale=.15,rotate=0] at (0,0) {pentominoF};
	\end{tikzpicture} ($\polfont F$) pentomino in Figure~\ref{fig:ex_sym}. 
	
	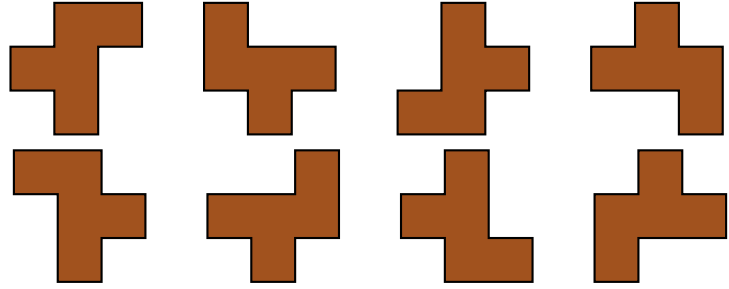
\begin{figure}[h]
		\centering
		\begin{tikzpicture}[scale=\scfa,baseline={(current bounding box.center)}]
		\pic[scale=\scfa,rotate=0] at (0,0) {pentominoF};.
		\end{tikzpicture}
		\qquad 
		\begin{tikzpicture}[scale=\scfa,baseline={(current bounding box.center)}]
			\pic[scale=\scfa,rotate=90] at (0,0) {pentominoF};.
		\end{tikzpicture}
		\qquad
		\begin{tikzpicture}[scale=\scfa,baseline={(current bounding box.center)}]
			\pic[scale=\scfa,rotate=180] at (0,0) {pentominoF};.
		\end{tikzpicture}
		\qquad
		\begin{tikzpicture}[scale=\scfa,baseline={(current bounding box.center)}]
			\pic[scale=\scfa,rotate=270] at (0,0) {pentominoF};.
		\end{tikzpicture}	\medskip
	
		\begin{tikzpicture}[scale=\scfa,baseline={(current bounding box.center)}]
			\pic[scale=\scfa,rotate=0] at (0,0) {pentominoF={-1}};.
		\end{tikzpicture}
		\qquad 
		\begin{tikzpicture}[scale=\scfa,baseline={(current bounding box.center)}]
			\pic[scale=\scfa,rotate=270] at (0,0) {pentominoF={-1}};.
		\end{tikzpicture}
		\qquad
		\begin{tikzpicture}[scale=\scfa,baseline={(current bounding box.center)}]
			\pic[scale=\scfa,rotate=180] at (0,0) {pentominoF={-1}};.
		\end{tikzpicture}
		\qquad 
		\begin{tikzpicture}[scale=\scfa,baseline={(current bounding box.center)}]
		\pic[scale=\scfa,rotate=90] at (0,0) {pentominoF={-1}};.
		\end{tikzpicture}
		\caption{The 8 potential symmetries of the pentomino $\polfont F$  are given by all the rotations of $0^\circ$, $90^\circ$, $180^\circ$, $270^\circ$, and their combinations with a mirror reflection.}
		\label{fig:ex_sym}
	\end{figure}
	We note that there are some redundancies, as some polyforms have symmetries. It is possible to fine-tune the model by taking them into account, but for the sake of simplicity, we do not.

	We fix an ordering of the polyforms and an ordering of their potential orientations coming from $S_T$. If $P=P_T^m$ is the number of polyforms given, then for every tile $l$ we define the binary variable matrix $z_l$ of dimensions $P \times s_T$. Here the entry $(p, r)$ is true if and only if the $p$-th polyform is placed in orientation $r\in S_T$ with the top left corner on $l$. 
	
	It will be useful to properly  define the notion of neighbourhood. We say two tiles are \emph{neighbours} if they share a border of lower dimension. For $d=2$, this means that two tiles are neighbours if they share an edge (a 1-dimensional border) or a vertex (a 0-dimensional border). For $d=3$, they are also neighbours if they share a face (a 2-dimensional border). The number of borders depends on the tessellation, as shown in Figure~\ref{fig:neighbourhood_tiles}.

	\begin{figure}[h]
		\centering
		\begin{tikzpicture}[scale=\scfa,baseline={(current bounding box.center)}]
			\draw[help lines] (0,0) grid (5,5);
			\draw[fill=\couleurb] (2,2) rectangle (3,3);
			\foreach \x/\y in {1/1,1/2,1/3,2/1,2/3,3/1,3/2,3/3}
			{
				\fill[pattern=north east lines, pattern color=teal] (\x,\y) rectangle (\x+1,\y+1) ;
			}
		\end{tikzpicture}
		\qquad
		\begin{tikzpicture}[hexa/.style= {shape=regular polygon,
				regular polygon sides=6,
				minimum size=\hexscfa cm, draw,
				inner sep=0,anchor=south,
			},scale=\hexscfa,,baseline={(current bounding box.center)}]
			
			\foreach \j in {0,...,4}{%
				\ifodd\j 
				\foreach \i in {0,...,3}{\node[hexa] (h\j;\i) at ({\j/2+\j/4},{(\i+1/2)*sin(60)}) 
					{}        
					;}        
				\else
				\foreach \i in {1,...,3}{\node[hexa] (h\j;\i) at ({\j/2+\j/4},{\i*sin(60)}) 
					{} 
					;}
				\fi}         
			\node[hexa, fill=\couleurb] (h2;2) at ({2/2+2/4},{2*sin(60)}) {}; 
			
			\foreach \x/\y in {1/2,1/3,2/1.5,2/3.5,3/2,3/3} {
				\node[hexa,pattern=north east lines, pattern color=teal] (h\x;\y)  at ({\x/2+\x/4},{(\y-1/2)*sin(60)}) {};
			}
			
		\end{tikzpicture}
		\qquad
 \begin{tikzpicture}[scale=\triscfa,baseline={(current bounding box.center)}]
\newcommand{\nbx}{7}
\newcommand{\nby}{7}

\clip (tri cs: x=2,y=0) -- (tri cs:x=5,y=0) -- (tri cs:x=5,y=3) -- (tri cs: x=3,y=5) -- (tri cs: x=0,y=5) -- (tri cs: x=0,y=2) -- cycle; 
\foreach \j in {-1,...,\the\numexpr\nby-1}
{\foreach \i in {-1,...,\the\numexpr\nbx-1} 
	{
}}
\foreach \j in {-1,...,\the\numexpr\nby-1}
{\foreach \i in {-1,...,\the\numexpr\nbx-1} 
	{
		\draw (tri cs: x=\i,y=\j) -- (tri cs:x=\i,y=\nby-1);
		\draw (tri cs: x=\i,y=\j) -- (tri cs:x=\nbx-1,y=\j);
		\draw (tri cs: x=\i,y=\j) -- (tri cs:x=\j,y=\i);
}    }
\draw[fill=\couleurb] (tri cs:x=3,y=2) -- 
(tri cs:x=3-1,y=2) -- 
(tri cs:x=3-1,y=2+1)-- cycle;
\draw[ pattern=north east lines, pattern color=teal] 
(tri cs:x=2,y=2) -- 
(tri cs:x=2-1,y=2) -- 
(tri cs:x=2-1,y=2+1)-- cycle;
\draw[ pattern=north east lines, pattern color=teal] 
(tri cs:x=2,y=3) -- (tri cs:x=2-1,y=3) --                                  (tri cs:x=2-1,y=3+1)-- cycle;      
\draw[ pattern=north east lines, pattern color=teal] 
(tri cs:x=3,y=3) -- (tri cs:x=3-1,y=3) --                                  (tri cs:x=3-1,y=3+1)-- cycle;    
\draw[ pattern=north east lines, pattern color=teal] 
(tri cs:x=3,y=1) -- (tri cs:x=3-1,y=1) --                                  (tri cs:x=3-1,y=1+1)-- cycle;      
\draw[ pattern=north east lines, pattern color=teal] 
(tri cs:x=4,y=1) -- (tri cs:x=4-1,y=1) --                                  (tri cs:x=4-1,y=1+1)-- cycle;   
\draw[ pattern=north east lines, pattern color=teal] 
(tri cs:x=4,y=2) -- (tri cs:x=4-1,y=2) --                                  (tri cs:x=4-1,y=2+1)-- cycle;   
\draw[ pattern=north east lines, pattern color=teal] 
(tri cs:x=2,y=3) -- (tri cs:x=2+1,y=3) --                                  (tri cs:x=2+1,y=3-1)-- cycle;    
\draw[ pattern=north east lines, pattern color=teal] 
(tri cs:x=2,y=2) -- (tri cs:x=2+1,y=2) --                                  (tri cs:x=2+1,y=2-1)-- cycle;    
\draw[ pattern=north east lines, pattern color=teal] 
(tri cs:x=3,y=2) -- (tri cs:x=3+1,y=2) --                                  (tri cs:x=3+1,y=2-1)-- cycle;    
\draw[ pattern=north east lines, pattern color=teal] 
(tri cs:x=1,y=2) -- (tri cs:x=1+1,y=2) --                                  (tri cs:x=1+1,y=2-1)-- cycle;      
\draw[ pattern=north east lines, pattern color=teal] 
(tri cs:x=1,y=3) -- (tri cs:x=1+1,y=3) --                                  (tri cs:x=1+1,y=3-1)-- cycle;      
\draw[ pattern=north east lines, pattern color=teal] 
(tri cs:x=1,y=4) -- (tri cs:x=1+1,y=4) --                                  (tri cs:x=1+1,y=4-1)-- cycle;      
\end{tikzpicture}
		\caption{A tile (completely filled in yellow) and its neighbours (marked with teal crosses) in the three regular tessellations of the plane.}
		\label{fig:neighbourhood_tiles}
	\end{figure}
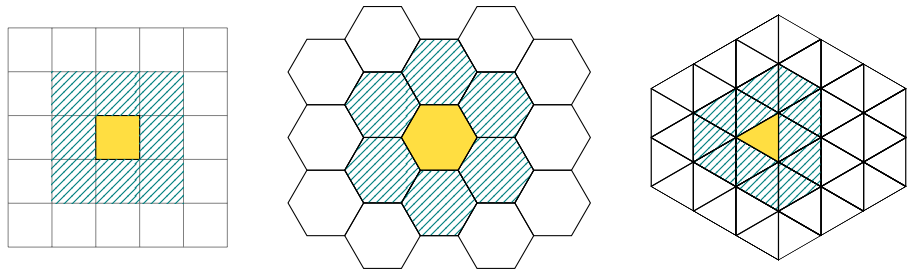
	
	Now we can focus on the constraints necessary to express our problem. The central constraint underpinning the idea of the whole model is that for every tile $l$ and all of its neighbours $l^\prime$, we have
	\begin{equation*}
		x_l \implies x_{l^\prime} \lor y_{l^\prime}.
	\end{equation*}
	
	So if a tile is in the interior of the fence, its neighbours also have to be in the interior or be on the fence itself. The implication can be translated to
	\begin{equation*}
		\neg x_l \lor x_{l^\prime} \lor y_{l^\prime},
	\end{equation*}
	and this can be converted to a linear inequality:
	\begin{equation*}
		(1 - x_l) + x_{l^\prime} + y_{l^\prime} \geq 1.
	\end{equation*}
	
	To ensure that the fence is closed, we require that all tiles on the exterior edge of our bounding region $R$ are neither interior nor part of the fence itself, so for all such tiles $l$ we just add $x_l = 0$ and $y_l = 0$ to our model.
	
	We also need to ensure that no tile is in the interior of the fence and also part of the fence, so for every tile $l$ we add
	\begin{equation*}
		x_l + y_l \leq 1.
	\end{equation*}
	
	For every tile $l$, we construct the set $A_l$ of all entries $z_{l^\prime, p, r}$, where placing the $p$-th polyform with its left upper corner on $l^\prime$ in orientation $r$ covers $l$. We can thus ensure that no two polyforms overlap by setting
	\begin{equation*}
		\sum_{j \in A_l} j \leq 1.
	\end{equation*}
	
	We also know that a tile is on the fence if and only if a polyform covers it, and that is converted in the following equation
	\begin{equation*}
		\sum_{j \in A_l} j = y_l.
	\end{equation*}
	
	As a last constraint, we need to ensure that every polyform is placed exactly once:
	\begin{equation*}
		\sum_{l \in R} \sum_{r\in S_T} z_{l, p, r} = 1.
	\end{equation*}
	
	Combining all the constraints discussed above, we get the following ILP model.
	\begin{equation}\label{eq:ILP}
		\begin{aligned}
			(\mathrm{ILP}) \quad
			& \text{maximize}   &&  \sum_{l\in R} x_l\\
			& \text{subject to} && \mathbf{x} \in \{0,1\}^{R}, \mathbf{y} \in \{0,1\}^{R}, \mathbf{z} \in \{0,1\}^{R, P, s_T}, \\
			& \quad\text{(Interior Spread)} &&  (1 - x_l) + x_{l^\prime} + y_{l^\prime} \geq 1,\quad \text {for all neighbour pairs } (l, l^\prime)\\
			& \quad\text{(Interior contained)} &&  x_l = 0, y_l = 0,\quad \text {for all edge tiles } l\\
			& \quad\text{(Fence and Interior)} &&  x_l + y_l \leq 1,\quad \text {for all tiles } l\\
			& \quad\text{(Overlap)} &&  \sum_{j \in A_l} j \leq 1,\quad \text {for all tiles } l\\
			& \quad\text{(Fence Implication)} &&  \sum_{j \in A_l} j = y_l,\quad \text {for all tiles } l\\
			& \quad\text{(Placed once)} &&  \sum_{l\in R} \sum_{r\in S_T} z_{l, p, r} = 1,\quad \text {for all polyforms } p.\\
		\end{aligned}
	\end{equation}
	
	An implementation of this model can be found here~\cite{LRMRgithub}.  It finds all valid solutions in less than an hour for the three tessellations of the plane; see Theorem~\ref{thm:tab:number_of_solutions}.

	 A small comment for the experts: this implementation does not enforce that the fence area is connected. The constraint can be added, but it slows down the ILP program. Heuristically, we can suppose that the maximal area will be connected, so removing the constraint should not matter, but formally it means we implement an upper bound search. Combining this upper bound with the example solution presented in Figure~\ref{fig:sol_max_pento} proves Theorem~\ref{thm:maxpento}.
	\begin{challenge}
		Implement the ILP model~\eqref{eq:ILP} for the pentominoes on your favourite ILP solver.
	\end{challenge}
	
	\begin{challenge}
		Recently, SAT solvers successfully attacked many different combinatorial problems and are able to generate verifiable proofs. Try to formally verify the optimality of the 128 solution using a modern (max)SAT solver.
	\end{challenge}
	
	\section{Conclusion and even more challenges}
	
	We hope that you found some inspiration in the challenges we proposed throughout the text. We encourage you to select a subset of them to share the next time you want to invite others to experience a beautiful piece of mathematics. In our experience, these challenges are well loved by people of all backgrounds, regardless of their prior mathematical knowledge. In fact, this approach fits within a broader framework of sharing contemporary mathematical research through games; an approach we warmly encourage you to explore, even within your own research~\cite{LRRR24}.

	\subsection{Three last challenges}
	To close this final chapter, we leave you with three final unsolved challenges.
	
	Challenge~\ref{challenge:hexomino_fence} highlights both the strengths and limitations of the ILP formulation~\eqref{eq:ILP}. We consider the hexomino fence problem, which involves 35 distinct hexominoes. Due to the combinatorial explosion of the fence configurations, the problem quickly becomes intractable for the solver to handle efficiently. Nonetheless, our solver was able to produce a fence that encloses an area of  1586, a value we believe to be very close to the global maximum, likely within 20 tiles (see Figure~\ref{fig:hexamono_sol_better}). Achieving this area of 1586 required approximately two months of computation using our algorithm. At present, identifying the true global maximum remains beyond reach given our current computational resources and methods.

	\begin{challenge}\label{challenge:hexomino_fence}
		Can you find a hexomino fence that encloses more than  1586 tiles? Our current best attempt, shown in Figure~\ref{fig:hexamono_sol_better}, encloses an area of 1586.
	\end{challenge}
	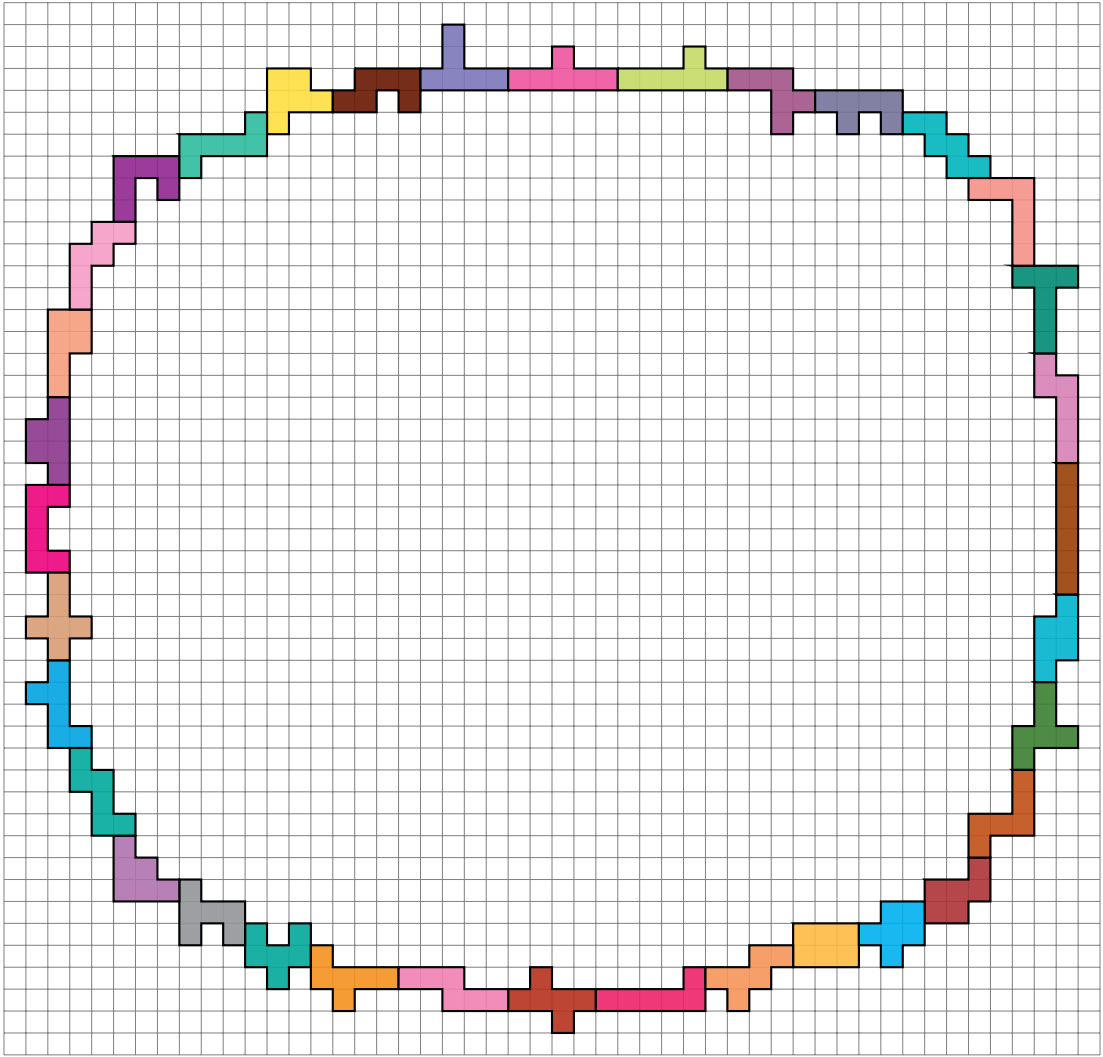
\begin{figure}[h]
		\centering
		\begin{tikzpicture}[scale=0.5*\scfa,yscale=-1]
			\draw[help lines] (-1,-2) grid (49,46);
			\draw[thick,fill=\couleura ,fill opacity=\opacite] (12,43) -- (11,43) -- (11,42) -- (10,42) -- (10,40) -- (11,40) -- (11,41) -- (12,41) -- (12,40) -- (13,40) -- (13,42) -- (12,42)-- cycle;
			\draw[thick,fill=\couleurb ,fill opacity=\opacite] (11,1) -- (11,4) -- (12,4) -- (12,3) -- (14,3) -- (14,2) -- (13,2) -- (13,1) -- cycle;
			\draw[thick,fill=\couleurc ,fill opacity=\opacite] (36,2) |- (37,3) |- (38,4) |- (39,3) |- (40,4) |- cycle;
			\draw[thick,fill=\couleurd ,fill opacity=\opacite] (47,25) |- (46,26) |- (47,29) |- (48,28)  |- cycle;
			\draw[thick,fill=\couleurF ,fill opacity=\opacite] (13,41) |- (14,43) |- (15,44) |- (17,43) |- (14,42) |- cycle;
			\draw[thick,fill=\couleurf ,fill opacity=\opacite] (32,1) |- (34,2) |- (35,4) |- (36,3) |- (35,2) |- cycle;
			\draw[thick,fill=\couleurg ,fill opacity=\opacite] (15,1) |- (14,2)  |- (16,3) |- (17,2) |- (18,3) |- cycle;
			\draw[thick,fill=\couleurh ,fill opacity=\opacite] (7,38) |- (8,41)  |- (9,40) |- (10,41) |- (8,39) |- cycle;
			\draw[thick,fill=\couleuri ,fill opacity=\opacite] (47,19) |- (48,25)  |- cycle;
			\draw[thick,fill=\couleurj ,fill opacity=\opacite] (46,29) |- (45,31)  |- (46,33) |- (48,32) |- (47,31)|- cycle;
			\draw[thick,fill=\couleurk ,fill opacity=\opacite] (22,43) |- (24,44)  |- (25,45) |- (26,44) |- (24,43) |- (23,42) |- cycle;
			\draw[thick,fill=\couleurl ,fill opacity=\opacite] (26,43) |- (31,44)  |- (30,42) |- cycle;
			\draw[thick,fill=\couleurM ,fill opacity=\opacite] (4,36) |- (7,39)  |- (6,38) |- (5,37) |- cycle;
			\draw[thick,fill=\couleurm ,fill opacity=\opacite] (40,3) |- (41,4)  |- (42,5) |- (44,6) |- (43,5) |- (42,4) |- cycle;
			\draw[thick,fill=\couleurN ,fill opacity=\opacite] (17,42) |- (19,43)  |- (22,44) |- (20,43) |- cycle;
			\draw[thick,fill=\couleurn ,fill opacity=\opacite] (46,14) |- (47,16)  |- (48,19) |- (47,15)|- cycle;
			\draw[thick,fill=\couleuro ,fill opacity=\opacite] (35,40) |- (38,42)   |- cycle;
			\draw[thick,fill=\couleurP ,fill opacity=\opacite] (1,12) |- (2,16)  |- (3,14) |- cycle;
			\draw[thick,fill=\couleurp ,fill opacity=\opacite] (1,16) |- (0,17)  |- (1,19) |- (2,20) |-  cycle;
			\draw[thick,fill=\couleurq ,fill opacity=\opacite] (41,38) |- (43,40)  |- (44,39) |- (43,37) |-  cycle;
			\draw[thick,fill=\couleurr ,fill opacity=\opacite] (1,28) |- (0,29)  |- (1,30) |- (3,32) |- (2,31) |- cycle;
			\draw[thick,fill=\couleurs ,fill opacity=\opacite] (2,32) |- (3,34)  |- (5,36) |- (4,35) |- (3,33) |- cycle;
			\draw[thick,fill=\couleurT ,fill opacity=\opacite] (45,10) |- (46,11)  |- (47,14) |- (48,11) |- cycle;
			\draw[thick,fill=\couleurt ,fill opacity=\opacite] (18,2) |- (19,1)  |- (20,-1) |- (22,1) |- cycle;
			\draw[thick,fill=\couleurU ,fill opacity=\opacite] (0,20) |- (2,24)  |- (1,23) |- (2,21) |- cycle;
			\draw[thick,fill=\couleuru ,fill opacity=\opacite] (4,5) |- (5,8)  |- (6,6) |- (7,7) |- cycle;
			\draw[thick,fill=\couleurv ,fill opacity=\opacite] (43,6) |- (45,7)  |- (46,10) |- cycle;
			\draw[thick,fill=\couleurW ,fill opacity=\opacite] (3,8) |- (2,9)  |- (3,12) |- (4,10) |- (5,9) |- cycle;
			\draw[thick,fill=\couleurw ,fill opacity=\opacite] (31,42) |- (32,43)  |- (33,44) |- (34,43) |- (35,42) |- (33,41)|- cycle;
			\draw[thick,fill=\couleurX ,fill opacity=\opacite] (1,24) |- (0,26)  |- (1,27) |- (2,28) |-  (3,27) |- (2,26) |- cycle;
			\draw[thick,fill=\couleurx ,fill opacity=\opacite] (38,40) |- (39,41)  |- (40,42) |- (41,41) |- (39,39) |- cycle;
			\draw[thick,fill=\couleurY ,fill opacity=\opacite] (22,2) |- (24,1) |- (25,0) |- (27,1) |- cycle;
			\draw[thick,fill=\couleury ,fill opacity=\opacite] (27,1) |- (32,2) |- (31,1) |- (30,0) |- cycle;
			\draw[thick,fill=\couleurZ ,fill opacity=\opacite] (7,4) |- (8,6)  |- (11,5) |- (10,3) |- cycle;
			\draw[thick,fill=\couleurz ,fill opacity=\opacite] (45,33) |- (43,35)  |- (44,37) |- (46,36) |- cycle;
		\end{tikzpicture}
		\caption{The biggest known fence for the hexomino fence on 28/08/2025, which encloses an area of 1586. Since then, the biggest fence has been found! We will give it in the next arXiv version, once it is published.}
		\label{fig:hexamono_sol_better}
	\end{figure}
	
	And what about the worst fence configurations for solving these fence challenges? That would mean taking the set of all fence configurations, finding the maximum area that can be enclosed by each one of them, and then the minimum of these maxima would give the answer to this question. We call this the \textit{Min-max fence problem}. We already know that there are fence configurations whose maxima are 127, others 126, and others 125 or less. So we know that this number is  bounded above by $125$.
	
	\begin{challenge}\label{challenge:minmax}
		Solve the pentomino Min-max fence problem.
	\end{challenge}
	
	\begin{hint}
		To find a better bound, it suffices to exhibit a fence configuration that cannot enclose an area larger than 124. We have \textit{human-generated} evidence suggesting that this min-max value should be greater than 115: during all of our workshop and outreach activities, participants consistently managed to construct fences enclosing more than 115 tiles, regardless of the fence configuration they were given or chose themselves. A \emph{mathematician}'s intuition on these types of problems, however, indicate that probably the search space is very sparse. This means that only few of the fence configurations will be ``bad''. So, an informed human, as we trust our reader to be by now, would be to try to find a particularly bad fence configuration and prove it cannot make a fence enclosing 115 tiles for example!
	\end{hint}

	Finally, all the work we have done so far has focused on polyforms in the plane. However, nothing prevents us from extending these ideas to higher dimensions. Instead of finding fences that divide the plane into two regions, we now consider fences that separate the $d$-dimensional space into two \emph{disconnected} regions. In higher dimensions, this disconnectedness that we are referring to becomes a more subtle notion. For instance, in the three-dimensional cubical tessellation, a cubical fence separates space into two polycubes that do not share any vertices; hence, they also do not share edges or faces.
	
	\begin{challenge}\label{challenge:pentocube}
		Construct all pentacubes, 3D polyforms made of five unit cubes, and use them to determine the largest volume that can be enclosed using these shapes.
	\end{challenge}

	\addcontentsline{toc}{section}{Acknowledgements}
	
	\section*{Acknowledgments}
	We would like to thank Dr. Erich Fuchs, who has been challenging students at math seminars with the pentomino problem for many years and thus provided the original inspiration for the project. We also wish to thank Tamara Sprinkle for her professional translation of the article~\cite{Ta78}.
	
	ALR wishes to acknowledge the hospitality of the Max Planck Institute for Mathematics in the Sciences during his visits and the support of ScaDS.AI Leipzig where he was employed during part of this research. Furthermore, ALR's research is funded by a postdoctoral research scholarship of the Fonds de Recherche du Québec -- Nature et Technologie [grant number 326641] and was funded by the Deutsche Forschungsgemeinschaft (DFG, German Research Foundation) under Germany's Excellence Strategy -- GZ 2047/1--2, Projekt-ID 390685813.
	
	 Finally, we thank Helmut Podhaisky and Mykhailo Lyader, as well as Günter Rote and his group Corbin Masur, Maximilian Prietzel, Matthes Norden, Frederik Safner, Lennard Scharein, and Daniel Yu for their interest in this work and sharing their solutions to the challenges. We  thank Günter Rote furthermore for his comments on the German version and catching inaccuracies and typos, which we corrected here.
	
	\printbibliography
	\addresseshere
	
	\newpage
	
	\appendix

	\section{Printable tessellation boards and pieces}\label{app:fences}
	
	\subsection{ The square tessellation board}
	\[
	\begin{tikzpicture}[scale=\scfa]
		\draw[help lines] (1,1) grid (20,20);
	\end{tikzpicture}
	\]	
\begin{center}
		\begin{tikzpicture}[scale = \scfa,baseline={(current bounding box.center)}]
		\draw  (0,5) node[left] {\ScissorRight};
		\pic[scale=\scfa] at (1,1) {tetrominoI};
		\draw (1.5,1) node[below] {\polfont i};
		\pic[scale=\scfa] at (4,1) {tetrominoL};
		\draw (4.5,1) node[below] {\polfont l};
		\pic[scale=\scfa] at (7,1) {tetrominoN};
		\draw (8.5,1) node[below] {\polfont n};
		\pic[scale=\scfa] at (11,1) {tetrominoO};
		\draw (12.5,1) node[below] {\polfont o};
		\pic[scale=\scfa] at (14,1) {tetrominoT};
		\draw (15.5,1) node[below] {\polfont t};
		\end{tikzpicture}	

	\begin{tikzpicture}[scale=\scfa]
	\draw  (-1,8) node[left] {\ScissorRight};
	\pic[scale=\scfa] at (0,0) {pentominoF};
	\draw (1.5,0) node[below] {\polfont F};
	\pic[scale=\scfa] at (4,0) {pentominoI};
	\draw (4.5,0) node[below] {\polfont I};
	\pic[scale=\scfa] at (8,0) {pentominoL};
	\draw (8.5,0) node[below] {\polfont L};
	\pic[scale=\scfa] at (12,0) {pentominoN};
	\draw (12.5,0) node[below] {\polfont N};
	\pic[scale=\scfa] at (15,0) {pentominoP};
	\draw (15.5,0) node[below] {\polfont P};
	\pic[scale=\scfa] at (19,0) {pentominoT};
	\draw (20.5,0) node[below] {\polfont T};
	\pic[scale=\scfa] at (0,6) {pentominoU};
	\draw (1.5,10) node[below] {\polfont U};
	\pic[scale=\scfa] at (4,6) {pentominoV};
	\draw (5.5,10) node[below] {\polfont V};
	\pic[scale=\scfa] at (8,6) {pentominoW};
	\draw (9.5,10) node[below] {\polfont W};
	\pic[scale=\scfa] at (12,6) {pentominoX};
	\draw (13.5,10) node[below] {\polfont X};
	\pic[scale=\scfa] at (16,5) {pentominoY};
	\draw (16.5,10) node[below] {\polfont Y};
	\pic[scale=\scfa] at (19,5) {pentominoZ};
	\draw (20.5,10) node[below] {\polfont Z};
	\end{tikzpicture}    
\end{center}
\newpage
	
\subsection{The hexagonal tessellation board}
\phantom{a}\vspace{3cm}
	\[
	\begin{tikzpicture}[baseline={(current bounding box.center)},hexa/.style= {shape=regular polygon,
		regular polygon sides=6,
		minimum size=\hexscfa cm, draw,
		inner sep=0,anchor=south,
	},scale=\hexscfa]
	
	\foreach \j in {0,...,10}{%
		\ifodd\j 
		\foreach \i in {-1,...,7}{\node[hexa] (h\j;\i) at ({\j/2+\j/4},{(\i+1/2)*sin(60)}) 
			{}        
			;}        
		\else
		\foreach \i in {0,...,7}{\node[hexa] (h\j;\i) at ({\j/2+\j/4},{\i*sin(60)}) 
			{} 
			;}
		\fi}      
	\end{tikzpicture}\]
	\phantom{a}\vspace{2cm}
	\begin{center}
			\begin{tikzpicture}[hexa/.style= {shape=regular polygon,
			regular polygon sides=6,
			minimum size=\hexscfa cm, draw,
			inner sep=0,anchor=south,
		},scale=\hexscfa]
		\draw  (-1,3) node[left] {\ScissorRight};
		\foreach \j in {0,...,3}{%
			\node[hexa,fill=\couleurhexa] (h0;\j) at ({0},{(\j+1/2)*sin(60)}) {};
		}
		\node[] (h0,-1) at (0,0) {\polfont{bar}};
		\node[hexa,fill=\couleurhexb] (h2;0) at ({2+0},{(0+1/2)*sin(60)}) {};
		\node[hexa,fill=\couleurhexb] (h2;1) at ({2+0},{(1+1/2)*sin(60)}) {};
		\node[hexa,fill=\couleurhexb] (h2;2) at ({2+0},{(2+1/2)*sin(60)}) {};
		\node[hexa,fill=\couleurhexb] (h3;2) at ({2+3/4},{2*sin(60)}) {};
		\node[] (h2,-1) at (2,0) {\polfont{pistol}};
		\node[hexa,fill=\couleurhexc] (h4;0) at ({4+0},{(0+1/2)*sin(60)}) {};
		\node[hexa,fill=\couleurhexc] (h4;1) at ({4+0},{(1+1/2)*sin(60)}) {};
		\node[hexa,fill=\couleurhexc] (h4;2) at ({4+0},{(2+1/2)*sin(60)}) {};
		\node[hexa,fill=\couleurhexc] (h5;3) at ({4+3/4},{3*sin(60)}) {};
		\node[] (h2,-1) at (4,0) {\polfont{worm}};
		\node[hexa,fill=\couleurhexd] (h6;0) at ({6+3/4},{(2)*sin(60)}) {};
		\node[hexa,fill=\couleurhexd] (h6;1) at ({6+0},{(0+1/2)*sin(60)}) {};
		\node[hexa,fill=\couleurhexd] (h6;2) at ({6+0},{(1+1/2)*sin(60)}) {};
		\node[hexa,fill=\couleurhexd] (h6;3) at ({6+3/4},{3*sin(60)}) {};
		\node[] (h2,-1) at (6,0) {\polfont{wave}};
		\node[hexa,fill=\couleurhexe] (h8;0) at ({8+3/4},{1*sin(60)}) {};
		\node[hexa,fill=\couleurhexe] (h8;1) at ({8+0},{(0+1/2)*sin(60)}) {};
		\node[hexa,fill=\couleurhexe] (h8;2) at ({8+0},{(2+1/2)*sin(60)}) {};
		\node[hexa,fill=\couleurhexe] (h8;3) at ({8+3/4},{2*sin(60)}) {};
		\node[] (h2,-1) at (8,0) {\polfont{arc}};
		\node[hexa,fill=\couleurhexf] (h0;0) at ({10+3/4},{1*sin(60)}) {};
		\node[hexa,fill=\couleurhexf] (h0;1) at ({10+0},{(0+1/2)*sin(60)}) {};
		\node[hexa,fill=\couleurhexf] (h0;2) at ({10+0},{(1+1/2)*sin(60)}) {};
		\node[hexa,fill=\couleurhexf] (h6;3) at ({10+3/4},{2*sin(60)}) {};
		\node[] (h2,-1) at (10,0) {\polfont{bee}};
		\node[hexa,fill=\couleurhexg] (h0;0) at ({13-3/4},{2*sin(60)}) {};
		\node[hexa,fill=\couleurhexg] (h0;1) at ({13+0},{(0+1/2)*sin(60)}) {};
		\node[hexa,fill=\couleurhexg] (h0;2) at ({13+0},{(1+1/2)*sin(60)}) {};
		\node[hexa,fill=\couleurhexg] (h6;3) at ({13+3/4},{2*sin(60)}) {};
		\node[] (h2,-1) at (13,0) {\polfont{propeller}};
		\end{tikzpicture}
	\end{center}
	
	\newpage
	\subsection{The triangle tessellation board}
	\phantom{a}\vspace{2cm}
	\[
	\begin{tikzpicture}[scale=\triscfa]
		\newcommand{\nbx}{15}
		\newcommand{\nby}{15}
		
		\clip (tri cs: x=6,y=0) -- (tri cs:x=14,y=0) -- (tri cs:x=14,y=8) -- (tri cs: x=8,y=14) -- (tri cs: x=0,y=14) -- (tri cs: x=0,y=6) -- cycle; 
		
		\foreach \j in {0,...,\the\numexpr\nby-1}
		{\foreach \i in {0,...,\the\numexpr\nbx-1} 
			{
		}}
		\foreach \j in {0,...,\the\numexpr\nby-1}
		{\foreach \i in {0,...,\the\numexpr\nbx-1} 
			{
				\draw (tri cs: x=\i,y=\j) -- (tri cs:x=\i,y=\nby-1);
				\draw (tri cs: x=\i,y=\j) -- (tri cs:x=\nbx-1,y=\j);
				\draw (tri cs: x=\i,y=\j) -- (tri cs:x=\j,y=\i);
		}    } 
	\end{tikzpicture}
	\]
	\phantom{a}\vspace{2cm}
	\begin{center}
				\begin{tikzpicture}[scale=\triscfa]
		\draw  (-2,5) node[left] {\ScissorRight};
		\pic[scale=\triscfa,rotate=60] at (0,2) {hexiaRhomboid};
		\node[] at (0,0) {\polfont{rhomboid}};
		\pic[scale=\triscfa,rotate=60] at (4,2) {hexiaCrook};
		\node[] at (4,0) {\polfont{crook}};
		\pic[scale=\triscfa,rotate=60] at (8,2) {hexiaCrown};
		\node[] at (8,0) {\polfont{crown}};
		\pic[scale=\triscfa,rotate=60] at (12,2) {hexiaSphinx};
		\node[] at (12,0) {\polfont{sphinx}};
		\pic[scale=\triscfa,rotate=60] at (16,2) {hexiaSnake};
		\node[] at (16,0) {\polfont{snake}};
		\pic[scale=\triscfa,rotate=60] at (20,2) {hexiaYacht};
		\node[] at (20,0) {\polfont{yacht}};
		\pic[scale=\triscfa,rotate=60] at (0,-4) {hexiaChevron};
		\node[] at (0,-6) {\polfont{chevron}};
		\pic[scale=\triscfa,rotate=60] at (4,-4) {hexiaSignpost};
		\node[] at (4,-6) {\polfont{signpost}};
		\pic[scale=\triscfa,rotate=60] at (8,-4) {hexiaLobster};
		\node[] at (8,-6) {\polfont{lobster}};
		\pic[scale=\triscfa,rotate=60] at (12,-4) {hexiaShoe};       
		\node[] at (12,-6) {\polfont{shoe}};
		\pic[scale=\triscfa,rotate=60] at (15.5,-4) {hexiaHexagon};
		\node[] at (15.5,-6) {\polfont{hexagon}};
		\pic[scale=\triscfa,rotate=60] at (20.5,-4) {hexiaButterfly};
		\node[] at (20,-6) {\polfont{butterfly}};
		\end{tikzpicture}
	\end{center}

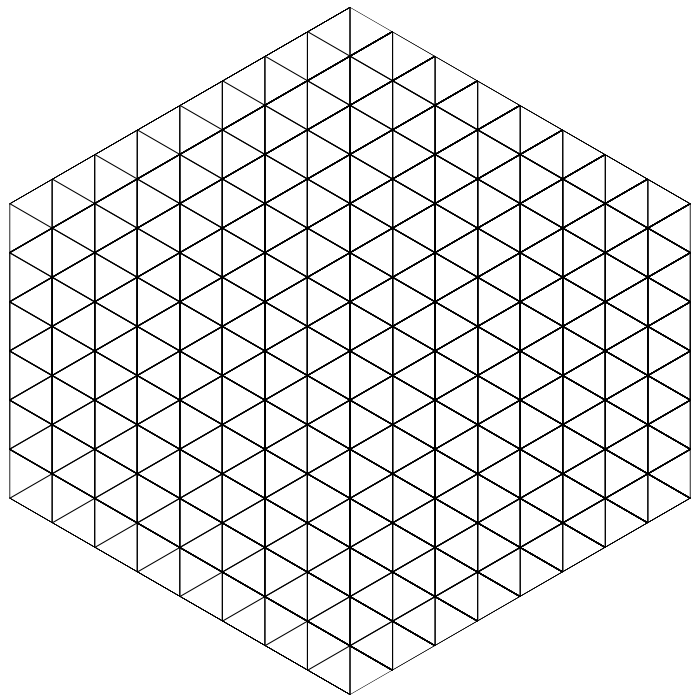
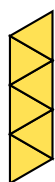
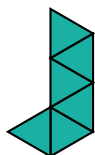
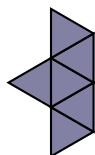
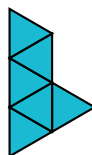
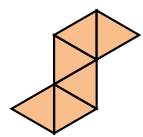
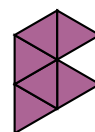
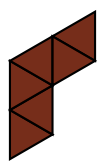
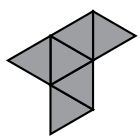
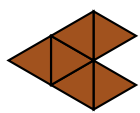
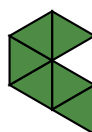
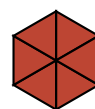
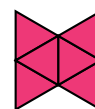
\end{document}